\documentclass{siamart251216}

\usepackage{amssymb}
\usepackage{microtype}
\usepackage{orcidlink}        %
\usepackage[shortlabels]{enumitem}
\newlist{thmenum}{enumerate}{2}
\setlist[thmenum,1]{%
  label=(\roman*), ref=(\roman*),
  leftmargin=*, labelindent=0pt, labelsep=0.25em, align=left,
  topsep=0.5ex, itemsep=0.25ex, parsep=0pt}
\setlist[thmenum,2]{%
  label=(\alph*), ref=(\alph*),
  leftmargin=*, labelindent=0pt, labelsep=0.25em, align=left,
  topsep=0.25ex, itemsep=0.25ex, parsep=0pt}
\usepackage{silence}
\usepackage{subcaption}      %
\usepackage{algpseudocode}   %
\usepackage{mathtools}
\usepackage{esint}
\usepackage{braket}
\usepackage[most]{tcolorbox}
\tcbset{
  thmbox/.style={
    enhanced, breakable,
    colback=blue!5!white,
    colframe=black,
    boxrule=0.4pt,
    arc=1pt,
    left=3pt, right=3pt, top=3pt, bottom=3pt,
    boxsep=2pt,
    before skip=3pt plus 1pt minus 1pt, after skip=\topsep,
  }
}

\numberwithin{equation}{section}
\numberwithin{figure}{section}

\newif\ifdraftnotes\draftnotestrue
\ifdraftnotes
  \usepackage[dvipsnames]{xcolor}
  \newcommand{\nb}[1]{{\color{RedOrange} Nick: #1}}
  \newcommand{\mpf}[1]{{\color{blue} (Michael: #1)}}
  \newcommand{\jb}[1]{{\color{OliveGreen} (Jim: #1)}}
\else
  \newcommand{\nb}[1]{}\newcommand{\mpf}[1]{}\newcommand{\jb}[1]{}
\fi

\usepackage{aliascnt}
\makeatletter
\newcommand{\siam@aliasthm}[2]{%
  \theoremstyle{plain}%
  \theoremheaderfont{\normalfont\sc}%
  \theorembodyfont{\normalfont\itshape}%
  \theoremseparator{.}%
  \theoremsymbol{}%
  \newaliascnt{#1}{theorem}%
  \@ifundefined{#1}%
    {\newtheorem{#1}[#1]{#2}}%
    {\renewtheorem{#1}[#1]{#2}}%
  \aliascntresetthe{#1}}
\newcommand{\siam@aliasrem}[2]{%
  \theoremstyle{plain}%
  \theoremheaderfont{\normalfont\itshape}%
  \theorembodyfont{\normalfont}%
  \theoremseparator{.}%
  \theoremsymbol{}%
  \newaliascnt{#1}{theorem}%
  \@ifundefined{#1}%
    {\newtheorem{#1}[#1]{#2}}%
    {\renewtheorem{#1}[#1]{#2}}%
  \aliascntresetthe{#1}}
\renewcommand{\newsiamthm}[2]{\siam@aliasthm{#1}{#2}}
\renewcommand{\newsiamremark}[2]{\siam@aliasrem{#1}{#2}}
\siam@aliasthm{lemma}{Lemma}
\siam@aliasthm{corollary}{Corollary}
\siam@aliasthm{proposition}{Proposition}
\siam@aliasthm{definition}{Definition}
\makeatother
\newsiamremark{remark}{Remark}
\newsiamthm{example}{Example}
\tcolorboxenvironment{theorem}{thmbox}
\tcolorboxenvironment{lemma}{thmbox}
\tcolorboxenvironment{proposition}{thmbox}
\tcolorboxenvironment{corollary}{thmbox}
\tcolorboxenvironment{definition}{thmbox}
\tcolorboxenvironment{remark}{thmbox}
\crefname{lemma}{Lemma}{Lemmas}\Crefname{lemma}{Lemma}{Lemmas}
\crefname{corollary}{Corollary}{Corollaries}\Crefname{corollary}{Corollary}{Corollaries}
\crefname{proposition}{Proposition}{Propositions}\Crefname{proposition}{Proposition}{Propositions}
\crefname{definition}{Definition}{Definitions}\Crefname{definition}{Definition}{Definitions}
\crefname{remark}{Remark}{Remarks}\Crefname{remark}{Remark}{Remarks}
\crefname{example}{Example}{Examples}\Crefname{example}{Example}{Examples}

\usepackage[numbers,sort]{natbib}

\crefname{enumi}{}{}
\crefname{enumii}{}{}
\crefname{enumiii}{}{}
\crefname{thmenumi}{}{}
\crefname{thmenumii}{}{}
\crefname{ALG@line}{Step}{Steps}
\Crefname{ALG@line}{Step}{Steps}

\crefformat{equation}{\textup{#2(#1)#3}}
\crefrangeformat{equation}{\textup{#3(#1)#4--#5(#2)#6}}
\crefmultiformat{equation}{\textup{#2(#1)#3}}{ and \textup{#2(#1)#3}}{, \textup{#2(#1)#3}}{ and \textup{#2(#1)#3}}
\crefrangemultiformat{equation}{\textup{#3(#1)#4--#5(#2)#6}}{ and \textup{#3(#1)#4--#5(#2)#6}}{, \textup{#3(#1)#4--#5(#2)#6}}{ and \textup{#3(#1)#4--#5(#2)#6}}

\makeatletter
\renewcommand{\ALG@step}{%
  \addtocounter{ALG@line}{1}%
  \global\def\@currentcounter{ALG@line}%
  \addtocounter{ALG@rem}{1}%
  \ifthenelse{\equal{\arabic{ALG@rem}}{\ALG@numberfreq}}%
    {\setcounter{ALG@rem}{0}\alglinenumber{\arabic{ALG@line}}}%
    {}%
}%
\makeatother

\newcounter{step}

\newcommand{\proofpart}[1]{Part~\ref{#1}.}
\newcommand{\partref}[1]{Part~\ref{#1}}

\newcommand{\R}{\mathbb{R}}
\newcommand{\Nzero}{\mathbb{N}_0} %

\newcommand{\T}{^\intercal}
\newcommand{\ones}{\mathbf{1}}

\newcommand{\rr}{\mathbb{R}}

\def\bB{\mathbb{B}}

\newcommand{\cC}{\mathcal{C}}

\newcommand{\cL}{\mathcal{L}}

\def\cU{\mathcal{U}}

\DeclarePairedDelimiterX{\norm}[1]{\lVert}{\rVert}{#1}

\DeclareMathOperator{\rint}{rint}
\DeclareMathOperator{\vexp}{\mathbf{exp}}
\DeclareMathOperator{\vlog}{\mathbf{log}}
\DeclareMathOperator{\logexp}{logexp}
\newcommand{\softmax}{\nabla\!\logexp}

\DeclareMathOperator{\interior}{int}

\DeclareMathOperator*{\argmin}{argmin}
\DeclareMathOperator{\dom}{dom}

\newcommand{\ip}[1]{\langle#1\rangle}
\newcommand{\bip}[2]{\left\langle #1,\, #2\right\rangle}
\newcommand{\map}[3]{#1: #2\rightarrow #3}
\newcommand{\rset}[2]{\left\{#1\left\vert\, #2\right.\right\}}

\newcommand{\lev}[2]{\mathrm{lev}_{#1}(#2)}
\newcommand{\half}{\mbox{\small $\frac{1}{2}$}}

\newcommand{\alf}{\alpha}

\def\talf{{\tilde\alf}}
\newcommand{\eps}{\varepsilon}
\newcommand{\balf}{\bar\alpha}
\def\by{{\bar y}}
\def\lam{\lambda}

\newcommand{\gam}{\gamma}
\newcommand{\del}{\delta}
\def\hdel{{\hat\del}}
\newcommand{\Del}{\Delta}
\def\Deln{{\Del_n}}

\newcommand{\sig}{\sigma}
\newcommand{\extrr}{\overline{\rr}}
\newcommand{\bN}{\mathbb{N}}

\newcommand{\btau}{\bar\tau}

\newcommand{\hx}{{\hat x}}

\newcommand{\tsum}{\textstyle\sum}

\newcommand{\mathL}{\mathsf{L}}
\newcommand{\mathM}{\mathsf{M}}

\newcommand{\tL}{{ \tilde \mathL }}

\newcommand{\hLD}{{ {\hat\mathL}_{\!\mbox{\tiny D}} }}
\newcommand{\LD}{{{\mathL}_{\!\mbox{\tiny D}}}}
\newcommand{\hM}{{ \hat\mathM }}
\def\cN{\mathcal{N}}
\def\tH{{\widetilde H}}
\def\R{\mathbb{R}}
\def\Rn{\R^n}
\definecolor{darkgreen}{rgb}{0,0.4,0}

\def\taubetamin{{\tau_{\beta{\min}}}}
\def\htaubetamin{{\hat\tau_{\beta{\min}}}}
\def\taubetamax{{\tau_{\beta{\max}}}}

\newcommand{\ymax}{{\mathsf{y}_{\beta{\max}}}}

\newcommand{\Fymax}{{\widehat{\mathsf{F}}_{y\max}}}
\newcommand{\Ftaumax}{{\widehat{\mathsf{F}}_{{\tau}\max}}}
\def\dmax{{\mathsf{d}_{\max}}}
\def\tdmax{{\tilde{\mathsf{d}}_{\max}}}
\def\rhomax{{\rho_{\max}}}

\def\dlogexp{\softmax_q} %
 
\def\hbeta{{\hat\beta}}

\newcommand{\Diag}[1]{{\mathrm{Diag}\left(#1\right)}}
\newcommand{\onorm}[1]{\left\| #1\right\|_1}
\newcommand{\inorm}[1]{\left\| #1\right\|_\infty}
\def\hx{{\hat x}}

\def\bh{{\bar h}}
\def\hh{{\hat h}}
\def\hP{{\widehat P}}

\makeatletter
\providecommand*{\diff}%
{\@ifnextchar^{\DIfF}{\DIfF^{}}} \def\DIfF^#1{%
\mathop{\mathrm{\mathstrut d}}%
        \nolimits^{#1}\gobblespace}

\def\gobblespace{%
        \futurelet\diffarg\opspace}
\def\opspace{%
        \let\DiffSpace\!%
        \ifx\diffarg(%
            \let\DiffSpace\relax
        \else
            \ifx\diffarg[%
               \let\DiffSpace\relax
            \else
               \ifx\diffarg\{%
                   \let\DiffSpace\relax
               \fi\fi\fi\DiffSpace}

\makeatother

\headers{Scale-shape dual Newton for entropic least squares}%
        {N.\ Barnfield, J.\,V.\ Burke, M.\,P.\ Friedlander, and T.\ Hoheisel}

\title{A Scale-Shape Dual Newton Method for\\Entropic Least Squares\thanks{Date: April 29, 2026.}}

\author{%
  Nicholas Barnfield\orcid{0009-0003-9895-8347}\thanks{Department of Statistics, Harvard University,
    Cambridge, MA, USA (\email{nbarnfield@g.harvard.edu}).}
  \and James V.\ Burke\orcid{0000-0002-0215-7306}\thanks{Department of Mathematics, University of
    Washington, Seattle, WA, USA (\email{jvburke01@gmail.com}).}
  \and Michael P.\ Friedlander\orcid{0000-0003-0222-5222}\thanks{Departments of Computer Science and
    Mathematics, University of British Columbia, Vancouver, BC, Canada
    (\email{mpf@cs.ubc.ca}).}
  \and \mbox{Tim Hoheisel}\orcid{0000-0002-0782-6302}\thanks{Department of Mathematics and Statistics,
    McGill University, Montreal, QC, Canada (\email{tim.hoheisel@mcgill.ca}).}%
}

\begin{document}

\maketitle

\begin{abstract}
We give a damped inexact Newton method for entropy-regularized least-squares on the nonnegative orthant that converges globally at a linear rate with $\mathcal{O}(\log\epsilon^{-1})$ iteration complexity, locally at a superlinear-to-quadratic rate, and is immune to the finite-precision overflow that limits classical dual solvers. A scale-shape decomposition of the primal---separating its scale from its direction---produces a dual with a nonsingular Jacobian. Objectives and Jacobians are evaluated through stable log-sum-exp and softmax primitives. Lambert~W bounds on the scale uniformly control the Jacobian's spectrum, from which both rates follow. The solution map is jointly Lipschitz in the data, regularization parameter, and reference measure, and extends continuously to the vanishing-regularization limit. Experiments on a problem from analytic continuation of quantum Monte Carlo data confirm the predicted overflow resilience and convergence behavior.
\end{abstract}



\section{Introduction}

Recovering a nonnegative density, spectrum, or intensity from noisy samples of a linear smoothing map is a recurring problem across applied mathematics. Instances include analytic continuation of quantum Monte Carlo data for warm dense matter~\cite{chunaDualFormulationMaximum2025,chunaEstimatesDynamicStructure2025,dornheim_ImaginarytimeCorrelationFunction_2023}, maximum-entropy image reconstruction in astronomy~\cite{Skilling_MaximumEntropyImage_1984,theeventhorizontelescopecollaborationFirstM87Event2019}, dual methods for entropy maximization in crystallography~\cite{decarreau_DualMethodsEntropy_1992}, and entropic optimal transport~\cite{cuturi_SemidualRegularizedOptimal_2018}. The common discrete form is a Fredholm integral equation of the first kind with a nonnegativity constraint, and is exponentially ill-posed in the sense that small perturbations of the data can cause large variations in the recovered solution~\citep{engl_RegularizationInverseProblems_1996}. Regularization is unavoidable. Among candidate regularizers, the Shore--Johnson axioms~\cite{shore_AxiomaticDerivationPrinciple_1980} and the Bayesian maximum-entropy tradition~\cite{jaynes_InformationTheoryStatistical_1957a} jointly single out the relative entropy as the unique inference procedure that avoids spurious correlations under missing information.

We study this archetype in the following finite-dimensional form. Given $A\in\rr^{m\times n}$, $b\in\rr^m$, $c\in\rr^n$, $q\in\rr^n_{++}$, and $\lambda>0$, consider the problem
\begin{equation}\label{eq:P_problem}
    \min_{x \in \rr^n}\left\{ \psi_p(x)
    \coloneqq \tfrac{1}{2\lambda}\,\norm{Ax-b}^2 + \ip{c,x} + g_q(x) \right\},
\end{equation}
where $q$ serves as a reference measure in the (unnormalized) relative entropy
\begin{equation} \label{eq:g_mu_definition}
    g_q(x) = \tsum_{j=1}^n x_j \log(x_j/q_j),
\end{equation}
with the conventions $0\cdot\log 0 = 0$ and $u\log u = \infty$ for $u<0$. Since $g_q$ is strictly convex, supercoercive, and lower semicontinuous~\cite{Bauschke2017}, \cref{eq:P_problem} has a unique minimizer~$x^\ast$; when $q$ has zeros, restricting to its support incurs no loss of generality.

The Fenchel dual of~\eqref{eq:P_problem} is unconstrained and smooth, and a long tradition of algorithmic work exploits this structure~\cite{ben-tal_RoleDualityOptimization_1988,borwein_ConvergenceBestEntropy_1991,borwein1992partially,decarreau_DualMethodsEntropy_1992,eriksson_NoteSolutionLarge_1980}. Any naive dual solver, however, runs into two independent obstacles: poor conditioning and arithmetic overflow. The dual Hessian is globally unbounded along sublevel sets, so first-order methods attain only $\mathcal{O}(\epsilon^{-1})$ iteration complexity~\citep{nesterov_LecturesConvexOptimization_2018} and standard Newton globalizations yield no global rate. The primal-from-dual recovery map $x = q\odot\vexp(A\T y - \ones - c)$, where $\vexp$ and $\odot$ denote the elementwise exponential and elementwise product, and $\ones$ the vector of ones, may overflow in finite-precision arithmetic long before the dual iterate~$y$ reaches the optimum.

Lifting the mass $\tau := \ip{\ones,x}$, determined by~$y$ in the classical dual, to an independent dual variable alongside~$y$ resolves both obstacles. The optimality conditions reduce to a nonlinear system $F(y,\tau)=0$ whose Jacobian $DF$ is everywhere nonsingular on $\rr^m\times\rr_{++}$. Lambert~W bounds on~$\tau$ along sublevel sets of the merit function $\rho := \norm{F}$ give uniform control of the extreme singular values of $DF$, hence a global $\mathcal{O}(\log\epsilon^{-1})$ iteration complexity for damped inexact Newton on~$\rho$ and a superlinear-to-quadratic local rate~(\cref{thm:convergence of inexact Newton 1}). Both $F$ and $DF$ are evaluated using the weighted log-sum-exp and its softmax gradient, which are stable under the standard max-shift.

\paragraph{Contributions} The scale-shape reformulation of~\eqref{eq:P_problem} yields a primal--dual characterization with a nonsingular Jacobian (\cref{sec:new_approach}). A damped inexact Newton iteration on the merit function $\rho = \norm{F}$ is proposed  (\cref{sec:alg2}). The convergence analysis (\cref{sec:bounds,sec:convergence}) culminates in \cref{thm:convergence of inexact Newton 1}. Joint Lipschitz perturbation bounds in $(b,\lambda,q)$ and the $\lambda\downarrow 0$ regularization-path limit follow in \cref{sec:perturbation}. Numerical experiments on a synthetic uniform-electron-gas problem (\cref{sec:experiments}) demonstrate overflow resilience, scale recovery from a fixed $\tau_0$, and the regularization-path behavior predicted by \cref{thm:joint_perturbation,thm:lam convergence}. Reference Python code reproducing all figures and convergence numbers is available at \url{https://doi.org/10.5281/zenodo.19889724}.

\subsection{Classical dual formulation} \label{sec:dual_prior}

The entropy-regularized least-squares problem~\eqref{eq:P_problem} admits an unconstrained dual formulation with the objective
\begin{equation} \label{eq:original_dual_obj}
        \psi_d(y) \coloneq \ip{b,y} - \tfrac{\lambda}{2}\norm{y}^2 - \ip{q, \vexp(A\T y - \ones - c)}.
\end{equation}
The next result specializes the infinite-dimensional theory of Decarreau~\cite{decarreau_DualMethodsEntropy_1992} to our finite-dimensional setting. The proof uses standard Fenchel--Rockafellar duality.

\begin{proposition}[Strong duality for entropy-regularized least squares] \label{thm:original_primal_dual}
    Assume $\lambda>0$ and $q\in\rr^n_{++}$. Then, strong duality holds with attainment:
    \begin{equation}\label{eq:original_strong_dual}
        \min_{x \in \rr^n_+} \psi_p(x) \;=\; \max_{y \in \rr^m}  \psi_d(y).
    \end{equation}
    Moreover, the optimal primal-dual pair $(x^\ast, y^\ast)$ is unique and satisfies
    \begin{equation} \label{eq:Primal-Dual-Opt}
        Ax^\ast + \lambda y^\ast = b
        \quad\text{and}\quad
        x^\ast = q \odot \vexp(A\T y^\ast - \ones - c).
    \end{equation}
\end{proposition}

\begin{proof}
Setting $\nabla[\ip{w,y} - g_q(w)] = 0$ yields the unique maximizer $w_i = q_i \exp(y_i - 1)$, whence $g_q^\ast(y) = \ip{q, \vexp(y - \ones)}$. The conjugate of $\tfrac{1}{2\lambda}\|(\cdot) - b\|^2$ is $\tfrac{\lambda}{2}\norm{\cdot}^2 + \ip{b,\cdot}$. Applying Fenchel--Rockafellar duality \cite[Cor.\ 31.2.1]{rockafellar-1970}, where qualification holds since $\dom(g_q + \ip{c,\cdot}) = \rr^n_+ \neq \emptyset$ and $\dom(\tfrac{1}{2\lambda}\|\cdot - b\|^2) = \rr^m$, yields the dual objective $\psi_d$ in~\eqref{eq:original_dual_obj} and strong duality~\eqref{eq:original_strong_dual}. The optimality conditions~\eqref{eq:Primal-Dual-Opt} follow from differentiating the dual.
\end{proof}

\Cref{thm:original_primal_dual} makes the unconstrained dual a convenient route to the primal solution: any gradient-based method applied to $\psi_d$ recovers $x^\ast$ at no additional cost via~\eqref{eq:Primal-Dual-Opt}, and the regularization parameter controls the residual directly through $\norm{Ax^\ast-b}=\lambda\norm{y^\ast}$. The obstacles described above persist, however: the unnormalized vector $q\odot\vexp(A\T y - \ones - c)$ necessarily enters $\nabla\psi_d$ and $\nabla^2\psi_d$, but admits no max-shift stabilization. Avoiding overflow requires lifting the scale, as developed in \cref{subsec:self_scaling_overview}, rather than rewriting $\psi_d$ in the log domain. 

\subsection{Scale-shape decomposition and the inexact Newton method}\label{subsec:self_scaling_overview}

The common source of the two obstacles is that, in the classical dual, the scale $\ip{\ones,x}$ and the direction $x/\ip{\ones,x}$ are both determined by the single variable~$y$ through the exponential map $x = q\odot\vexp(A\T y - \ones - c)$. Separating these is the purpose of the scale-shape decomposition.

Every $x\in\rr^n_+$ admits the \emph{scale-shape decomposition} $x = \tau p$, where $\tau = \ip{\ones,x}\ge 0$ and $p = x/\tau$ lies in the probability simplex $\Delta_n := \set{x\in\rr^n_+ | \ip{\ones,x}=1}$; this partitions the nonnegative orthant $\rr^n_+ = \bigcup_{\tau\ge 0}\tau\Delta_n$ into scaled simplices, that is, $\Del_n$ is a compact base for the convex cone $\R^n_+$~\cite{NonLin_COnvOpt}. The corresponding scale-shape dual objective is
\begin{equation*}\label{eq:dual_overview}
  \phi_d(y,\tau) = \ip{b,y} - \tfrac{\lambda}{2}\norm{y}^2 - \tau\logexp_q(A\T y - c) + \tau\log\tau,
\end{equation*}
where $\logexp_q(u):=\log\ip{q,\vexp(u)}$ is the weighted log-sum-exp and its gradient is the weighted softmax $\dlogexp(u):=q\odot\vexp(u)/\ip{q,\vexp(u)}\in\Delta_n$. Setting $\nabla_{(y,\tau)}\phi_d = 0$ yields the nonlinear system $F(y,\tau)=0$ whose Jacobian is everywhere nonsingular on $\rr^m\times\rr_{++}$. On each sublevel set of $\rho=\norm{F}$ the scale~$\tau$ is bracketed by computable Lambert~W bounds (\cref{lem:level cpt}), which translate into uniform bounds on the extreme singular values of~$DF$ (\cref{lem:DF inv bd}), furnishing the global and local convergence rates of \cref{thm:convergence of inexact Newton 1}. When $\tau$ is known, the algorithm reduces to inexact Newton in~$y$ alone (see \cref{rem:fixed_scale}).

\subsection{Related work}\label{sec:related_work}

Prior work on the classical dual~\eqref{eq:original_dual_obj} can be organized by how the scale parameter is handled. \citet{eriksson_NoteSolutionLarge_1980} applies Newton's method directly to $\psi_d$ but establishes neither convergence nor stability; the scale-shape system $F(y,\tau)=0$ instead admits the global linear rate and $\mathcal{O}(\log\epsilon^{-1})$ iteration bound proved in \cref{sec:convergence}. In the \emph{known-scale} setting of entropy-regularized optimal transport, \citet{cuturi_SemidualRegularizedOptimal_2018} controls exponential growth by solving a log-domain dual with L-BFGS \citep{liu_LimitedMemoryBFGS_1989}. The present method instead treats the \emph{unknown} scale $\tau$ as an independent dual variable.

For unknown scale, \citet{Skilling_MaximumEntropyImage_1984}, and \citet{bryan_MaximumEntropyAnalysis_1990} propose Newton methods that require a well-conditioned surrogate for the forward operator. The Jacobian analysis of \cref{sec:alg2} removes this requirement; \citet{chunaDualFormulationMaximum2025} gives a detailed comparison in the analytic-continuation setting. Adaptive gradient methods \citep{malitsky_AdaptiveGradientDescent_2020,lu_AcceleratedFirstorderMethods_2023} exploit local smoothness but produce rates governed by trajectory-dependent constants unavailable \emph{a priori}.

\section{Dual under scale-shape separation} \label{sec:new_approach}

We now formalize the scale-shape decomposition outlined in \cref{subsec:self_scaling_overview}. Writing $x = \tau p$ with $\tau = \ip{\ones,x}$ and $p \in \Delta_n$, the reformulated problem becomes
\begin{equation}\label{eq:primal_problem_scale_shape}
  \min_{x\in\rr^n_+} \psi_p(x) = \min_{\tau\ge0} \min_{x \in \tau\Delta_n} \phi_p(x,\tau),
\end{equation}
where the primal objective is
\begin{equation}\label{eq:phi primal}
  \phi_p(x,\tau) \coloneq \tfrac{1}{2\lambda}\|Ax -b\|^2 + \ip{c,x} + (g_q + \delta_{\tau \Delta_n})(x).
\end{equation}
Since $g_q$ is strictly convex and coercive, so is $\phi_p(\cdot,\tau)$, and the inner primal problem $\min_{x\in\tau\Deln}\phi_p(x,\tau)$ is attained uniquely for every $\tau\ge 0$.

Dualizing~\eqref{eq:primal_problem_scale_shape} yields a saddle-point formulation. The corresponding first-order optimality conditions drive the Newton method of \cref{sec:alg2}.

\begin{proposition}[Scale-shape dual characterization] \label{thm:primal_dual_duality}
Let $\lambda > 0$ and $q \in \rr^n_{++}$. Define the dual objective
\begin{equation} \label{eq:dual_func}
   \phi_d(y, \tau) \coloneq
   \ip{b,y} - \tfrac{\lambda}{2}\norm{y}^2 - \tau \logexp_q( A\T y - c ) + \tau \log \tau,
\end{equation}
for $\tau\ge 0$, with the convention $0\log 0 = 0$ (see \eqref{eq:g_mu_definition}). The following hold:
\begin{thmenum}
\item \label{item:duality}
For every $\tau\ge 0$, strong duality with attainment holds:
\[
\min_{x\in\tau\Delta_n}\phi_p(x,\tau) = \max_{y\in\rr^m}\phi_d(y,\tau).
\]
Consequently, the primal problem admits the saddle-point representation
\begin{equation}\label{eq:minmax}
  \min_{x \in \rr^n_+} \psi_p(x) = \min_{\tau\ge0}\max_{y \in \rr^m} \phi_d(y,\tau).
\end{equation}

\item \label{item:strong_concavity}
For every $\tau\ge 0$, the map $y \mapsto \phi_d(y, \tau)$ is $(-\lambda)$-strongly concave. In particular, the maximizer in~$y$ in \partref{item:duality} is unique.

\item \label{item:solutions}
Let $x^\ast = \argmin_{x \in \rr^n_+}  \psi_p(x)$ and define
$\tau^\ast := \ip{\ones, x^\ast}$. 
Then $\tau^\ast>0$, and the characterizations
\begin{subequations}\label{eq:PrimalDualMap_simplex}
\begin{align}
    x^\ast &= \tau^\ast p(y^\ast) \in \rr^n_{++},  \label{eq:PrimalDualMap_simplex_x}\\
    y^\ast &= \tfrac{1}{\lambda}(b - Ax^\ast),\\
    \tau^\ast &= \ip{q, \vexp(A\T y^\ast - \ones - c)}, \label{eq:PrimalDualMap_simplex_tau}
\end{align}
\end{subequations}
hold, and $(y^\ast,\tau^\ast)$ is the unique minimax point of~\eqref{eq:minmax}.

\item \label{item:uniqueness}
For $\tau>0$, define $\map{F}{\rr^m\times\rr_{++}}{\rr^m\times\rr}$ by
\begin{equation} \label{eq:dual_grad}
    F(y, \tau) := \nabla_{y,\tau}\phi_d(y,\tau) =
    \begin{bmatrix}
     F_y(y,\tau)  \\[3pt]
     F_\tau(y,\tau)
   \end{bmatrix}
   :=
    \begin{bmatrix}
     b - \lambda y - \tau A p(y)  \\[3pt]
    -\logexp_q(A\T y - c) + \log \tau + 1
   \end{bmatrix},
\end{equation}
where $p(y):=\dlogexp(A\T y-c)$. Then $(y^\ast, \tau^\ast)$ is the unique zero of~$F$.

\item \label{item:DF-invertible}
The Jacobian $DF(y,\tau)$ is nonsingular for all $(y, \tau) \in \rr^m \times \rr_{++}$.
\end{thmenum}
\end{proposition}

\begin{proof}

\noindent
\proofpart{item:duality}
Fix $\tau\ge 0$ and define $k_\tau:=g_q+\delta_{\tau\Delta_n}$. The inner problem in~\eqref{eq:primal_problem_scale_shape} can be written as
\[
\min_{x\in\rr^n}\big\{h(b-Ax) + \ip{c,x} + k_\tau(x)\big\},
\qquad
h:=\tfrac{1}{2\lambda}\|\cdot\|^2.
\]
Since $h$ is continuous on $\rr^m$ and $\dom k_\tau=\tau\Delta_n\neq\emptyset$, Fenchel--Rockafellar duality \citep[Example~11.41]{RockWets98} gives strong duality with attainment and the dual objective
\[
\max_{y\in\rr^m}\Big\{\ip{b,y}-h^\ast(y)-(k_\tau+\ip{c,\cdot})^\ast(A\T y)\Big\}.
\]
Here $h^\ast=\tfrac{\lambda}{2}\|\cdot\|^2$, and the substitution $w=\tau p$ in the definition of $k_\tau^\ast$ gives
\[
k_\tau^\ast(w) = \tau (g_q+\delta_{\Delta_n})^\ast(w) - \tau\log\tau = \tau\logexp_q(w) - \tau\log\tau,
\]
using the standard simplex conjugate $(g_q+\delta_{\Delta_n})^\ast=\logexp_q$ \cite[\S4.4.10]{firstorderBeck}. Since $(k_\tau+\ip{c,\cdot})^\ast(w)=k_\tau^\ast(w-c)$, the dual is $\max_y\phi_d(y,\tau)$; combined with~\eqref{eq:primal_problem_scale_shape} this yields~\eqref{eq:minmax}.

\medskip\noindent
\proofpart{item:strong_concavity}
The Hessian of $y\mapsto\phi_d(y,\tau)$ is $\nabla_{yy}^2\phi_d(y,\tau) = -\lambda I - \tau\,A\,S(y)\,A\T$, where
\begin{equation}\label{eq:S(y)}
S(y) := \Diag{p(y)} - p(y)p(y)\T = \nabla^2\logexp_q(A\T y - c) \succeq 0
\end{equation}
by \cref{lem:logexp_properties}\cref{item:logexp_smooth}. Therefore $\nabla_{yy}^2\phi_d(y,\tau)\preceq -\lambda I$ for all $\tau\ge 0$, proving $(-\lambda)$-strong concavity of $y\mapsto\phi_d(y,\tau)$ and uniqueness of the maximizer in~$y$ for each fixed~$\tau$.

\medskip\noindent
\proofpart{item:solutions}
Proposition~\ref{thm:original_primal_dual} gives the unique $(x^\ast,y^\ast)$ satisfying
\[
Ax^\ast+\lambda y^\ast=b,
\qquad
x^\ast=q\odot\vexp(A\T y^\ast-\ones-c),
\]
and hence $x^\ast\in\rr^n_{++}$. Therefore $\tau^\ast=\ip{\ones, x^\ast}>0$. Setting $p^\ast:=x^\ast/\tau^\ast$ gives $p^\ast\in\rint\Delta_n$ and
\[
p^\ast
=\frac{q\odot\vexp(A\T y^\ast-c)}{\ip{q,\vexp(A\T y^\ast-c)}}
=\dlogexp(A\T y^\ast-c),
\qquad
\tau^\ast=\ip{q,\vexp(A\T y^\ast-\ones-c)},
\]
which yields~\eqref{eq:PrimalDualMap_simplex}.
Since $x^\ast\in\tau^\ast\Delta_n$, strong duality at $\tau=\tau^\ast$ gives
\[
\min_{x\in\tau^\ast\Delta_n}\phi_p(x,\tau^\ast)=\max_{y\in\rr^m}\phi_d(y,\tau^\ast)=\psi_p(x^\ast).
\]
A direct computation gives $\nabla_y\phi_d(y,\tau) = b - \lambda y - \tau A p(y)$.
Substituting $y^\ast = \tfrac{1}{\lambda}(b - Ax^\ast)$ and $x^\ast = \tau^\ast p^\ast$ yields
$\nabla_y\phi_d(y^\ast,\tau^\ast) = 0$, so by \partref{item:strong_concavity}, $y^\ast$ is the unique maximizer at~$\tau^\ast$, and hence $(y^\ast,\tau^\ast)$ attains~\eqref{eq:minmax}. Uniqueness of the minimax point follows because $x^\ast$ is unique and $\tau^\ast=\ip{\ones, x^\ast}$ is determined by $x^\ast$, while $y^\ast$ is the unique maximizer of $\phi_d(\cdot,\tau^\ast)$ by \partref{item:strong_concavity}.

\medskip\noindent
\proofpart{item:uniqueness}
The identities in~\eqref{eq:PrimalDualMap_simplex} imply $F(y^\ast,\tau^\ast)=0$. Conversely, if $F(\hat y,\hat\tau)=0$ with $\hat\tau>0$, set $\hat x:=\hat\tau\,p(\hat y)$. Then $b=\lambda\hat y+A\hat x$, and the scalar equation gives $\hat\tau=e^{-1}\ip{q,\vexp(A\T\hat y-c)}$, whence $\hat x=q\odot\vexp(A\T\hat y-\ones-c)$. Hence $(\hat x,\hat y)$ satisfies~\eqref{eq:Primal-Dual-Opt}, so \cref{thm:original_primal_dual} gives $(\hat x,\hat y,\hat\tau)=(x^\ast,y^\ast,\tau^\ast)$.

\medskip\noindent
\proofpart{item:DF-invertible}
Set $H := -\lambda I - \tau A S(y)A\T \prec 0$ and $g := -A p(y)$. The Jacobian
\begin{equation}  \label{eq:dual_hess}
    DF(y,\tau)=\nabla^2 \phi_d(y, \tau) =
    \begin{bmatrix}
      H & g \\[3pt]
      g\T &  1/\tau
     \end{bmatrix}.
\end{equation}
By the Schur complement formula,
  $\det (DF) = \det (H)\cdot  \left( \frac{1}{\tau} - \ip{g,H^{-1} g} \right)$,
and therefore $DF$ is nonsingular because $H^{-1}\prec 0$ implies $\ip{g,H^{-1} g} \le 0 < 1/\tau$.
\end{proof}

The following properties of $\logexp_q$ are used in the proof above and throughout the paper.

\begin{lemma}[Properties of $\logexp_q$]\label{lem:logexp_properties}
Let $q \in \rr^n_{++}$. Then,
\begin{thmenum}
\item \label{item:logexp_smooth}
$\logexp_q \in \cC^\infty(\rr^n)$ is convex with Hessian
\begin{equation*} \label{eq:hessian-logexp}
    \nabla^2\logexp_q(u) = \Diag{\dlogexp(u)} - \dlogexp(u)\dlogexp(u)\T \succeq 0
\end{equation*}
satisfying the norm bound $\norm{\nabla^2\logexp_q(u)} \le 1/2$. In particular, $\dlogexp$ is $(1/2)$-Lipschitz.
\item \label{item:logexp_simplex}
$\dlogexp$ maps $\rr^n$ into the relative interior of $\Delta_n$.
\item \label{item:logexp_composed}
For $A \in \rr^{m \times n}$ and $c \in \rr^n$, the map $y \mapsto\logexp_q(A\T y - c)$
is $A_{\max}$-Lipschitz, and $p(y)=\dlogexp(A\T y - c)$ is $(\half\norm{A})$-Lipschitz.
\end{thmenum}
\end{lemma}

\begin{proof}
\ref{item:logexp_smooth}. Write $\pi = \dlogexp(u)$ and $\pi_j$ for its $j$-th component. Then $\pi_j = q_j e^{u_j}/\ip{q,\vexp(u)}$ is $\cC^\infty$, so $\logexp_q$ is  $\cC^\infty(\rr^n)$. The Hessian $S(u) = \Diag{\pi} -\pi\pi\T$ is positive semidefinite because, for any $n$-vector $v$, $v\T S(u)\,v = \sum_j \pi_j v_j^2 - \bigl(\sum_j \pi_j v_j\bigr)^2 \ge 0$ by the Cauchy--Schwarz inequality applied to $\pi$.
Since $S_{ii} = \pi_i(1-\pi_i)$ and $\sum_{j\neq i}|S_{ij}| = \pi_i\sum_{j\neq i}\pi_j = \pi_i(1-\pi_i)$, Geršgorin's theorem \cite[Theorem 6.1.1]{horn_MatrixAnalysis_2017} gives $\theta \le 2\pi_i(1-\pi_i) \le 1/2$ for every eigenvalue $\theta$ of $S$, so $\norm{S(u)} \le 1/2$.

\ref{item:logexp_simplex}. Each component $\pi_j = q_j e^{u_j}/\ip{q,\vexp(u)} > 0$ and $\sum_j \pi_j = 1$.

\ref{item:logexp_composed}. Again from~\ref{item:logexp_simplex}, $p(y)=\dlogexp(A\T y - c) \in \Delta_n$, and hence
\[
\|\nabla_y \logexp_q(A\T y - c)\| = \|A p(y)\| \leq A_{\max}
\]
which yields the stated Lipschitz constant for $y \mapsto\logexp_q(A\T y - c)$. Next, the map $y \mapsto A\T y - c$ is $\norm{A}$-Lipschitz and, by~\ref{item:logexp_smooth}, $\dlogexp$ is $(1/2)$-Lipschitz. The Lipschitz constant for $p(y)$ then follows from the composition of Lipschitz functions.
\end{proof}

\section{Damped inexact Newton on the scale-shape dual} \label{sec:alg2}

\Cref{sec:new_approach} reduced the primal problem to solving $F(z)=0$, where $z=(y,\tau)\in\rr^m\times\rr_{++}$ bundles the dual and scale variables. This section develops a damped inexact Newton iteration for that equation. The analysis rests on three lemmas. First, the merit function $\rho=\norm{F}$ has compact sublevel sets whose $\tau$-projections are bounded away from~$0$ and~$\infty$, with explicit constants (\cref{lem:level cpt}). Second, $DF$ is uniformly nonsingular on each level set, and $\norm{DF^{-1}}$ is controlled by those $\tau$-bounds (\cref{lem:DF inv bd}). Third, the damped backtracking line search terminates and produces a strictly decreasing sequence of merit values (\cref{lem:alg well-defined}). \Cref{subsec:inexact} states the algorithm. The fixed-scale specialization in which $\tau$ is prescribed appears at the end of the section. The convergence theorem in \cref{sec:convergence} uses these three lemmas directly.

\subsection{Merit function and level set compactness}

An inexact Newton method is applied to compute the unique root $(y^*,\tau^*)$ of the map $F$ defined in~\eqref{eq:dual_grad}; see \cref{thm:primal_dual_duality}\ref{item:uniqueness}. Step sizes are
chosen to guarantee descent in the merit function
\begin{equation*}\label{eq:rho_def}
\rho(y,\tau) = 
\begin{cases}
\|F(y, \tau)\|& \text{if $\tau>0$,}\\
+\infty  &\text{otherwise}.
\end{cases}
\end{equation*}
The function $\rho$ is continuous on $\rr^m\times\rr_{++}$ but nonsmooth, and requires maintaining $\tau$ positive throughout the iterations.

The convergence analysis in \cref{sec:convergence} also requires the iterates to remain in a sublevel set $\lev{\rho}{\beta}:=\rset{(y,\tau)}{\rho(y,\tau)\le\beta}$ on which $DF$ is uniformly invertible. By~\eqref{eq:dual_hess}, $\norm{DF(y,\tau)^{-1}}$ depends on upper and lower bounds on~$\tau$ together with a bound on $\norm{y}$. The level set must therefore keep $\tau$ bounded away from $0$ and $\infty$ and $\norm{y}$ bounded. \Cref{lem:level cpt} below describes all three bounds explicitly.

The bounds are expressed in terms of the problem data
\begin{equation}\label{eq:level-constants}
c_{\min} := \min_{i} c_i, \quad
c_{\max} := \max_{i} c_i, \quad
A_{\max} := \max_i \|A_i\|, \quad
q_{\min} := \min_i q_i,
\end{equation}
where $A_i$ denotes the $i$th column of $A$.

They also involve the Lambert $W$ function~\cite{CGHJK1996}, which inverts the transcendental equations arising from the $\tau\log\tau$ term in~\eqref{eq:dual_func} and is the unique solution of
\begin{equation}\label{eq:Lambert}
W(\xi) e^{W(\xi)} = \xi
\end{equation}
for $\xi \geq 0$; equivalently, $W$ is the inverse of $t \mapsto t e^t$ on $[0,\infty)$. Moreover, $W(\xi) > 0$ for $\xi > 0$, since $W^{-1}(\theta) = \theta e^\theta > 0$ for $\theta > 0$.

\begin{lemma}[Level-set compactness and $(y,\tau)$ bounds]  \label{lem:level cpt}
The level set $\lev{\rho}{\beta}$ %
is compact for every $\beta\geq0$. Moreover, computable bounds on the projections
of $\lev{\rho}{\beta}$ are given by
\begin{align*}
  \taubetamax
    &:= \sup\rset{\tau}{\exists\, y\in\rr^m:\,(y,\tau)\in\lev{\rho}{\beta}}
    \le B/W(B e^\theta),
    \tag{i}\label{item:taubetamax}\\
  \taubetamin
    &:= \inf\rset{\tau}{\exists\, y\in\rr^m:\,(y,\tau)\in\lev{\rho}{\beta}}
    \ge \lambda(A_{\max})^{-2}\, W(\zeta)>0,
    \tag{ii}\label{item:taubetamin}\\
  \ymax
    &:= \sup\rset{\norm{y}}{\exists\, \tau>0:\,(y,\tau)\in\lev{\rho}{\beta}}
    \le \tfrac{1}{\lambda}(\norm{b} + \taubetamax A_{\max} + \beta),
    \tag{iii}\label{item:ymax}
\end{align*}
where
\begin{subequations}\label{eq:level-aux}
\begin{align}
  \theta &:= 1 - \beta - \log\onorm{q} + c_{\min}, \label{eq:theta}\\
  B &:= \tfrac{1}{4\lambda}(\beta + \norm{b})^2, \label{eq:B}\\
  \zeta &:= \tfrac{1}{\lambda} A_{\max}^2 q_{\min}\,
    \exp\bigl({-}1 - \beta - c_{\max} - \tfrac{1}{\lambda} A_{\max}(\norm{b}+\beta) \bigr). \label{eq:zeta}
\end{align}
\end{subequations}
\end{lemma}

\begin{proof}
Let $F_y$ and $F_\tau$ be the components of $F(y,\tau)$ as given in \eqref{eq:dual_grad}, let $\beta > 0$ and suppose $(y, \tau) \in \lev{\rho}{\beta}$.
Recall from \cref{thm:primal_dual_duality}\cref{item:uniqueness} that $p := p(y) \in \rint \Delta_n$,
and set $\Sigma := \sum_{i=1}^n q_i \exp{(A_i\T y - c_i)}$.
In particular, we have 
\begin{equation} \label{eq:A_max_bound}
|A_i\T y|\leq A_{\max}\|y\| \enspace \text{for}\enspace  i=1,\dots m,\quad \text{and} \quad
\norm{Ap}\le \textstyle\sum_{i=1}^np_i\norm{A_i}\le A_{\text{max}}.
\end{equation}

\proofpart{item:taubetamax} From the definition of $p$,
\[
\log p_i = \log q_i + A_i\T y - c_i - \log \Sigma
\]
for $i = 1, \dots, n$. Multiply by $p_i$, sum over $i$, and rearrange to obtain
\begin{equation} \label{eq:logSigma}
    \log \Sigma = \ip{Ap,y} - \ip{c,p} - \textstyle\sum_{i=1}^n p_i \log(p_i/q_i)
\end{equation}
because $\sum_{i=1}^n p_i = 1$. From the definition of $F_\tau$, $\log \Sigma = \log \tau + 1 - F_\tau(y, \tau)$. Insert this into \eqref{eq:logSigma} and rearrange to obtain
\begin{equation} \label{eq:balance}
    \ip{Ap,y} = \textstyle\sum_{i=1}^n p_i \log(p_i/q_i) + \ip{c, p}
    + \log \tau  + 1 - F_\tau(y, \tau).
\end{equation}
On the other hand, we find
\begin{equation} \label{eq:Fy_ip_y}
    \ip{y, F_y(y, \tau)} = - \lambda \|y\|^2 + \ip{b, y} - \tau \ip{Ap,y}.
\end{equation}
Insert \eqref{eq:balance} into \eqref{eq:Fy_ip_y} and rearrange to obtain 
\begin{equation} \label{eq:tau_eqlty_relation}
    \tau \left( \log \tau + 1 - F_\tau(y, \tau) +  \textstyle\sum_{i=1}^n p_i \log \frac{p_i}{q_i} 
    + \ip{c, p} \right) =  - \lambda \|y\|^2 + \ip{b, y} -  \ip{y, F_y(y, \tau)}. 
\end{equation}
Since $(y, \tau) \in \lev{\rho}{\beta}$ it follows that $|F_\tau(y, \tau)| \leq \beta$ and 
$\|F_y(y, \tau)\| \leq \beta$, and as $p\in \Delta_n$ one has 
$\ip{c, p} \geq c_{\min}$. 
Moreover, the log-sum inequality \cite[Thm.~2.7.1]{CoverThomas2006}
gives 
\[
\textstyle\sum_{i=1}^n p_i \log \frac{p_i}{q_i} \geq \left( \sum_{i=1}^n p_i \right) \log \left(\frac{\sum_{i=1}^n p_i}{\sum_{i=1}^n q_i}\right) = 1 \cdot \log\left( \frac{1}{\sum_{i=1}^n q_i}\right)
=-\log\onorm{q}.
\]
Use \eqref{eq:theta} and 
place these inequalities into \eqref{eq:tau_eqlty_relation} to obtain
\begin{align*}
    \tau \left( \log \tau + \theta \right) 
    &\leq  - \lambda \|y\|^2 + \ip{b, y}  - \ip{y, F_y(y, \tau)}\\
    &\leq  - \lambda \|y\|^2 + \|b\| \|y\|  + \beta \|y\|  \\
    &\leq \sup_{t \geq 0}\ \{- \lambda t^2 + (\|b\| + \beta)t\} \\
    &= \tfrac1{4\lambda}(\|b\| + \beta)^2 \equiv B.
\end{align*} 
Set $u := \log\tau$ and $v := u + \theta$. The constraint becomes
$e^{v-\theta} v \le B$, i.e., $v e^v \le B e^\theta$.
Since $t \mapsto t e^t$ is strictly increasing on $[-1,\infty)$, it follows that $v \le W(B e^\theta)$.
Using the identity $e^{W(t)} = t/W(t)$, we have
\[
  \tau = e^u = e^{v-\theta} \le \frac{e^{W(Be^\theta)}}{e^\theta} = \frac{Be^\theta/W(Be^\theta)}{e^\theta} = \frac{B}{W(Be^\theta)},
\]
which establishes \partref{item:taubetamax}.

\proofpart{item:taubetamin} Since
$\| {-}\lambda y + b - \tau A p\|= \|F_y(y, \tau)\|  \leq \beta$,
the reverse triangle inequality and~\eqref{eq:A_max_bound} yield
\begin{equation}\label{eq:y-tau-bound}
  \|y\| \leq \tfrac{1}{\lambda}(\|b\| + \tau A_{\max} + \beta).
\end{equation}
Moreover, as $|F_\tau(y,\tau)| \leq \beta$, we have
$- \log \Sigma + \log \tau + 1 \geq - \beta$, and so
\begin{equation}\label{eq:tauEst1}
\tau \geq e^{-1 - \beta} \Sigma.
\end{equation}
Then, as $|A_i\T y|\leq A_{\max}\|y\|$, we have
\begin{equation}\label{eq:tauEst2}
\Sigma \geq q_{\min} \exp(- A_{\max} \|y\| - c_{\max}).
\end{equation}
Combining \eqref{eq:y-tau-bound}, \eqref{eq:tauEst1}, and~\eqref{eq:tauEst2}, we arrive at
\[
\tau \geq
q_{\min} \exp{\left(-1 -\beta - c_{\max} - \tfrac1\lambda A_{\max}(\|b\| + \tau A_{\max} + \beta )\right)}.
\]
Or, equivalently,
\begin{equation}\label{eq:tau bd1}
 \tfrac1\lambda A_{\max}^2 \tau \exp{(A^2_{\max} \tau/\lambda)} \geq \tfrac1\lambda A_{\max}^2 q_{\min} \exp{\left(-1 -\beta - c_{\max} - \tfrac1\lambda A_{\max}(\|b\| + \beta )\right)}
\end{equation}
for any $(y,\tau)\in \lev{\rho}{\beta}$.
Setting $w := \frac1\lambda A_{\max}^2\tau$, equation~\eqref{eq:tau bd1} takes the form $w e^w \geq \zeta$, where $\zeta$ is defined by~\eqref{eq:zeta}. Since $W$ is increasing on $\rr_+$ with $W(w e^{w})=w$ (cf.~\eqref{eq:Lambert}), we obtain
\[
\tfrac1\lambda A_{\max}^2\tau = w = W(w e^w)\ge W(\zeta),
\]
or equivalently, $\tau \ge \lambda W(\zeta)/A_{\max}^2$,
which establishes \partref{item:taubetamin}.

\proofpart{item:ymax} Combining~\eqref{eq:y-tau-bound} with \partref{item:taubetamax} gives
\begin{equation*}\label{eq:ymax}
  \|y\| \leq \tfrac{1}{\lambda}(\|b\| + \tau A_{\max} + \beta) \leq
  \tfrac1\lambda(\|b\| +  \taubetamax A_{\max} + \beta),
\end{equation*}
which establishes \partref{item:ymax}.

The compactness of $\lev{\rho}{\beta}$ now follows from combining \partref{item:taubetamax}--\partref{item:ymax} for boundedness, and the continuity of $\rho$ on the open domain $\rr^m \times \rr_{++}$ for closedness.
\end{proof}

The bound on $\taubetamin$ in \cref{lem:level cpt} is computable from the problem data. In what follows, $\htaubetamin$ denotes any positive lower estimate of $\taubetamin$; for instance, the closed-form value $\lambda W(\zeta)/A_{\max}^2$ from \cref{lem:level cpt}\ref{item:taubetamin} suffices. \Cref{alg:inexact Newton} below uses $\htaubetamin$ to safeguard $\tau$ away from~$0$.

\subsection{Inexact Newton algorithm}\label{subsec:inexact}

\Cref{alg:inexact Newton} realizes the inexact Newton iteration for $F(z)=0$. Each step computes an inexact Newton direction $d^k$, shortens it to a provisional length $\bar\alpha_k\le 1$ that keeps the updated scale $\tau_k+\bar\alpha_k\Delta\tau_k$ sufficiently positive, and then backtracks along $d^k$ until the merit function $\rho$ decreases by a fraction of the residual norm.

Two design choices deserve comment. First, \cref{step:provisional} caps a negative scale step so that $\tau_k+\bar\alpha_k\Delta\tau_k \ge \half\htaubetamin$; the factor of one-half gives the subsequent backtracking room to operate without driving $\tau$ outside $\dom\rho$. Second, the acceptance test in \cref{step:backtrack} uses the nonsmooth merit $\rho=\norm{F}$ rather than the standard $\tfrac12\norm{F}^2$. \Cref{rem:merit fnc} shows that any step satisfying the conventional Armijo condition for $\tfrac12\norm{F}^2$ also satisfies the test in \cref{step:backtrack}, so unit steps are readily accepted near the Newton direction.

\begin{algorithm}[t]
\caption{Inexact Newton with nonsmooth damping}
\label{alg:inexact Newton}
\begin{algorithmic}[1]
\Statex \textbf{Input:} Tolerance $\varepsilon\ge0$, backtracking parameters $\mu,\gamma\in(0,1)$, $\bar\eta\in[0,1)$, and $\{\eta_k\}_{k\ge 0}\subset[0,\bar\eta]$.
\State Initialize $z^0=(y^0,\tau_0)\in \R^{m}\times \R_{++}$; 
set $k\gets0$
and $\beta_0=\rho(z^0)$; choose $\beta>\beta _0$. Let 
$\htaubetamin>0$ be a lower estimate of $\taubetamin$ as in Lemma \ref{lem:level cpt}\ref{item:taubetamin}.
\State\label{step:newton}  \textbf{If} $\rho(z^k)=0$ \textbf{then} return $(y^k,\tau_k)$. \textbf{Else} find $d^k \in \rr^m \times \rr$ such that
    \begin{equation} \label{eq:inexact_newton_step}
        \| F(z^k)+DF(z^k) d^k\| \leq \eta_k \|F(z^k)\|  \quad {\rm where} \quad d^k = \begin{bmatrix}
            \Delta y^k \\
            \Delta \tau_k 
        \end{bmatrix}.
    \end{equation}
\State\label{step:provisional} \textbf{Compute provisional step.} Compute $\bar\alpha_k$ so that $\tau_k+\bar\alpha_k\,\Delta\tau_k \ge \half\htaubetamin>0$:
  \[
    \bar\alpha_k \gets
    \begin{cases}
      1 & \text{if } \tau_k+\Delta \tau_k \ge \half\htaubetamin, \\
      (\half \htaubetamin-\tau_k)/\Del\tau_k & \text{otherwise}.
    \end{cases}
  \]
\State\label{step:backtrack} \textbf{Backtracking for acceptance.}
\[
\alpha_k \gets \max_{\ell\in \Nzero}\Bigl\{\bar\alpha_k\,\gamma^{\,\ell}\;\Bigm|\;
\rho(z^k+\bar\alpha_k\,\gamma^{\,\ell} d^k)
\le \rho(z^k) + \mu\,\bar\alpha_k\,\gamma^{\,\ell}\,(\eta_k-1)\,\rho(z^k)\Bigr\}.
\]
\State  \textbf{Update iterates.}\label{step:update_iterates}\begin{equation*}
       z^{k+1} = z^k + \alpha_k d^k,
       \quad
       p^{k+1} = \dlogexp(A\T y^{k+1} - c),
       \quad
       x^{k+1} = \tau_{k+1} p^{k+1}.
    \end{equation*}
\State  \textbf{Stopping criterion.} \textbf{If} $\rho(z^{k+1})\le \varepsilon$ \textbf{then} return $(y^{k+1},\tau_{k+1})$; \textbf{else} set $k\gets k+1$ and \textbf{go to} \cref{step:newton}.
\end{algorithmic}
\end{algorithm}

The primal solution is recovered as $x=\tau\,\dlogexp(A\T y-c)$ from the returned pair $(y,\tau)$.

\begin{remark}[Choice of merit function $\rho$]\label{rem:merit fnc}
The smooth merit function $\psi(z) = \frac{1}{2}\|F(z)\|^2$ leads a conventional backtracking line search
\begin{equation}\label{eq:conventional_armijo}
    \psi(z^k + \alpha d^k) \leq \psi(z^k) + \mu \alpha \ip{\nabla \psi(z^k), d^k}.
\end{equation}
where $\mu \in (0,1/2)$ is the backtracking parameter and $\alpha$ is the step size \cite[Chapter 9.2]{boyd2004convex}.
Because $\nabla\psi(z) = DF(z)\T F(z)$,
\begin{align*}
     \ip{\nabla \psi(z^k), d^k} &= \ip{F(z^k), DF(z^k)d^k + F(z^k)} - \|F(z^k)\|^2 \\
     &\leq \|F(z^k)\|(\|DF(z^k)d^k + F(z^k)\| - \|F(z^k)\|) \\
     &\leq (\eta_k - 1) \rho(z^k)^2,
\end{align*}
where the last inequality uses the inexact Newton condition~\eqref{eq:inexact_newton_step}. 
Because $\psi = \tfrac12\rho^2$, any step size $\alpha$ satisfying \eqref{eq:conventional_armijo} gives
\[
\rho(z^{k} + \alpha d^k)^2 \leq \rho(z^k)^2 + 2 \mu \alpha (\eta_k - 1) \rho(z^k)^2
=(1-2 \mu \alpha (1-\eta_k ))\rho(z^k)^2.
\]
The basic inequality $\sqrt{1 - \theta} \leq 1 - \theta/2$ for $\theta \in [0,1]$, applied with $\theta = 2 \mu \alpha (1-\eta_k) \le 2\mu < 1$ (valid as $\alpha \le 1$, $\eta_k \in [0,1)$, and $\mu < \tfrac{1}{2}$), yields
\[
\rho(z^{k} + \alpha d^k) \leq  \rho(z^k) [1 - \mu \alpha (1-\eta_k)].
\]
This matches the backtracking condition in \cref{step:backtrack} of \cref{alg:inexact Newton}. The step-size rule therefore admits any step size that satisfies the conventional Armijo condition~\eqref{eq:conventional_armijo}.
\end{remark}

\subsection{Well-posedness of the algorithm}

Two algorithmic issues remain: existence of the inexact Newton step \eqref{eq:inexact_newton_step}, and finite termination of the backtracking rule in \cref{step:backtrack}. \Cref{lem:DF inv bd} establishes the first and collects spectrum bounds on $DF$ used in \cref{sec:convergence}. \Cref{lem:alg well-defined} describes the second.

\begin{lemma}[Jacobian bounds and Lipschitz continuity]\label{lem:DF inv bd}
Let $0 < \tau_{l} < \tau_{u} < \infty$. Parts~\ref{item:singular-bounds}--\ref{item:DF inverse} give pointwise bounds on $DF(z)$; parts~\ref{item:strip-bounds}--\ref{item:level-set-bounds} specialize these to $\tau$-strips and to sublevel sets of~$\rho$.
\begin{thmenum}
\item \label{item:singular-bounds}
  Singular value bounds. For $z = (y,\tau) \in \rr^m \times \rr_{++}$, 
  the singular values of $DF(z)$ satisfy
  \begin{equation}\label{eq:singular value bounds}
    \min\{\lam, 1/\tau\} \!\le\! \sig_{\min}(DF(z))
    \!\le\! \sig_{\max}(DF(z))
   \! \le\! A_{\max} + \max\{\lam + \half\tau\norm{A}^2,\, 1/\tau\},
  \end{equation}
  where $A_{\max}$ is defined in \eqref{eq:level-constants}.
  In particular,
  $\norm{DF(z)^{-1}}\le \max\{\lam^{-1},\, \tau\}$.
\item \label{item:DF inverse}
$DF(z)$ inverse. For all $z\in\rr^m\times\rr_{++}$,
\begin{equation*}\label{eq:DF inverse}
\begin{aligned}
DF(z)^{-1}&=\begin{bmatrix}
-\tH(z)^{-1}&-\tau \tH(z)^{-1}Ap(y)\\ -\tau p(y)\T A\T \tH(z)^{-1} &
\tau -\tau^2 p(y)\T A\T \tH(z)^{-1}Ap(y)
\end{bmatrix}
\\ &=
\begin{bmatrix}
    I& 0\\ \tau p(y)\T A\T & 1
\end{bmatrix}
\begin{bmatrix}
    -\tH(z)^{-1}& 0\\ 0& \tau
\end{bmatrix}
\begin{bmatrix}
    I&\tau Ap(y)\\ 0&1
\end{bmatrix},
\end{aligned}
\end{equation*}
where $z=(y,\tau)$ and $\tH(z):=[\lam I+\tau A\Diag{p(y)}A\T ]$ with $p(y)$ and $S(y)$ being the continuous functions of $y\in\rr^m$ defined in \eqref{eq:dual_grad} and \eqref{eq:S(y)}, respectively.
Moreover,
\[
    \lim_{\stackrel{y\rightarrow \by}{\tau\downarrow 0}}DF(y,\tau)^{-1}=-\lam^{-1}\begin{bmatrix} I&0\\0&0\end{bmatrix} \text{ for all } \by\in\rr^m.
\]
\item \label{item:strip-bounds}
  Bounds on $\tau$-strips.
  We have
 \[
     \mathL(\tau_{l}, \tau_{u})\!:=\!
     \sup\rset{\norm{DF(z)}}{z \in \rr^m \!\times\! [\tau_{l}, \tau_{u}]} 
     \!\le\! %
     A_{\max} + \max\left\{\lam + \frac{\tau_{u}\norm{A}^2}2, \frac1{\tau_{l}}\right\}
\]
and
\[      
\mathM(\tau_u):=\sup\rset{\norm{DF(z)^{-1}}}
      {z \in \rr^m \times (0, \tau_{u}]} 
      \le %
      \max\{\lam^{-1}, \tau_{u}\}.
   \]

\item \label{item:lipschitz}
  Lipschitz continuity.
  $F$ is Lipschitz continuous on $\rr^m \times [\tau_{l}, \tau_{u}]$
  with constant $\mathL(\tau_{l}, \tau_{u})$.

\item \label{item:level-set-bounds}
  Level-set bounds.
  For $\beta > 0$, let $\taubetamin$ and $\taubetamax$ be as in Lemma~\ref{lem:level cpt}.
  Then, $F$ is Lipschitz continuous on $\rr^m \times [\half\taubetamin, 2\taubetamax]$
  with constant $\mathcal{L}(\beta) := \mathL(\half\taubetamin, 2\taubetamax)$,
  and, for all $z\in\lev{\rho}{\beta}$, $\norm{DF(z)^{-1}} \le  \mathM(2\taubetamax)$.
  Moreover, $\lev{\rho}{\beta} \subset \interior(\rr^m \times [\half\taubetamin, 2\taubetamax])$
  by Lemma~\ref{lem:level cpt}.
\end{thmenum}
\end{lemma}

\begin{proof} 
Part~\ref{item:singular-bounds}.
Recall from~\eqref{eq:dual_hess} that $DF(z)$ has the saddle-point structure defined by blocks (suppressing the $y$ and $\tau$ arguments for brevity)
$H = -\lambda I - \tau A S A\T$ and $g = -A p$,
where $p = p(y) \in \rint(\Delta_n)$
and $S = \Diag{p} - pp\T$
with $\norm{S}\le\half$ by Lemma 
\ref{lem:logexp_properties}\ref{item:logexp_smooth}.

To obtain the upper bound on $\sigma_{\max}(DF)$,
decompose $DF$ (cf.~\eqref{eq:dual_hess}) as
\begin{equation}\label{eq:DF decomposition}
DF = M_1 + M_2, \enspace\text{with}\enspace
 M_1 = \begin{bmatrix}H & 0\\ 0 & 1/\tau\end{bmatrix} \quad\text{and}\quad
 M_2 = \begin{bmatrix}0 & g\\ g\T & 0\end{bmatrix}.
\end{equation}
Because $M_2$ is symmetric and rank-2 with eigenvalues $\pm\norm{g}$, we have $\norm{M_2}=\norm{g}=\norm{Ap}\le A_{\max}$ as in \eqref{eq:A_max_bound}. Therefore,
\[
\sigma_{\max}(DF)=\norm{DF}\le \norm{M_1}+\norm{M_2}
\le\max\left\{\lam+\half\tau\norm{A}^2,\, 1/\tau\right\}+A_{\max}.
\]

We now derive the lower bound on $\sigma_{\min}(DF)$. 
Cauchy's interlacing theorem~\cite[Theorem 4.3.17]{horn_MatrixAnalysis_2017} applied to the principal submatrix $H$ implies
\[
\lambda_1(DF) \le \lambda_1(H) \le \lambda_2(DF) \le \cdots \le \lambda_m(H) \le \lambda_{m+1}(DF).
\]
Since $H = -\lambda I - \tau ASA\T$ and $\tau ASA\T \succeq 0$, we have $H \preceq -\lambda I$. Thus, by the above interlacing properties, the leftmost (smallest) eigenvalue of $DF$ has magnitude at least~$\lambda$.
For the rightmost (largest) eigenvalue, the Schur complement of $H$ in $DF$ equals $1/\tau - \ip{g, H^{-1}g} \ge 1/\tau$ and is positive because $H \prec 0$. By inertia additivity~\cite[Theorem~1]{haynsworthDeterminationInertiaPartitioned1968}, $DF$ has exactly one positive eigenvalue. The inclusion principle~\cite[Theorem~4.3.28]{horn_MatrixAnalysis_2017} applied to the $1\times 1$ principal submatrix $[1/\tau]$ of $DF$ establishes that this eigenvalue is at least $1/\tau$. Summarizing,
$\sigma_{\min}(DF) \ge \min\{\lambda,\,1/\tau\}$.
Hence $DF(z)^{-1}$ exists with
$\norm{DF(z)^{-1}}\le \max\{\lam^{-1},\,\tau\}$.

Part~\ref{item:DF inverse}. The formulas for $DF(z)^{-1}$ are established through matrix
multiplication. The limit formula follows directly from the formulas for the inverse. 

Part~\ref{item:strip-bounds} follows by taking the supremum and infimum of the bounds in~\eqref{eq:singular value bounds} over the strip $\rr^m \times [\tau_{l}, \tau_{u}]$,
where the supremum in the definition of $\mathM(\tau_u)$ is well defined due to the limit
formula in Part~\ref{item:DF inverse}.

Part~\ref{item:lipschitz} holds since $F$ is continuously differentiable and $\norm{DF(z)} \le \mathL(\tau_{l}, \tau_{u})$ uniformly on the strip.

Part~\ref{item:level-set-bounds} specializes parts~\ref{item:strip-bounds} and~\ref{item:lipschitz} to the strip $[\frac{1}{2}\taubetamin, 2\taubetamax]$; the inclusion $\lev{\rho}{\beta} \subset \interior(\rr^m \times [\frac{1}{2}\taubetamin, 2\taubetamax])$ is immediate from \cref{lem:level cpt}.
\end{proof}

By ensuring the Jacobian is never singular, \Cref{lem:DF inv bd} establishes that \cref{step:newton} of \cref{alg:inexact Newton} is
implementable. The following lemma addresses the termination of the backtracking line search in \cref{step:backtrack} and establishes the strict monotonic decrease of the merit function along the iterates.   

\begin{lemma}[Backtracking termination and monotonic decrease]\label{lem:alg well-defined} The sequence $\{z^k\}$ generated by \cref{alg:inexact Newton} is well defined. In particular, the following hold:
\begin{thmenum}
    \item there exists $\ell_k\in\Nzero$ such that $\alf_k:=\bar\alf_k\gam^{\ell_k}$ satisfies the Armijo condition in \cref{step:backtrack};
    \item $\rho(z^{k+1})<\rho(z^k)< \beta$.
\end{thmenum}
\end{lemma}

\begin{proof}
We analyze the directional derivative of $\rho$ to establish that backtracking terminates.
For any closed proper convex function 
    $\map{h}{\rr^m}{\extrr}$, the directional derivative
    $h'(x;d)$ exists for all
    $x\in\dom(h)$ in all directions $d\in\rr^m$ and is given by
    $h'(x;d) = \inf_{t > 0} \tfrac{1}{t}(h(x + td) - h(x))$
\cite[Theorem~23.1]{rockafellar-1970}.
Following \citet[Section 2]{burkeSecondOrderNecessary1987}, 
    for $f:=g\circ H$ where 
    $\map{g}{\rr^m}{\rr}$ is 
    convex and $\map{H}{\cU}{\rr^m}$ is continuously differentiable at $x\in\cU$ with $\cU\subset\rr^n$ open, 
    we have
    $f'(x;d)$ exists for all $d\in\rr^n$ with   
    \begin{equation*}\label{eq:hH_dd}
         f'(x;d)= g'(H(x);DH(x)d)
         \le g(H(x)+DH(x)d)-g(H(x)).
    \end{equation*}

The merit function $\rho(z)=\norm{F(z)}$ has this form with $g = \norm{\cdot}$, and $H = F$, and $\cU=\R^m\times \R_{++}$, so
    \begin{equation*}\label{eq:h_dd}
        \rho'(z^k;d^k)
        =\inf_{t>0}\frac{\|F(z^k) + tDF(z^k)d^k\| - \|F(z^k)\|}{t}
       \leq  \|F(z^k) + DF(z^k)d^k\| - \|F(z^k)\|.
\end{equation*}
Here, the right-hand side is finite by \cref{step:newton} which, in turn, is well-defined by \cref{lem:DF inv bd}, and the former also yields
    \begin{equation}\label{eq:phi ne zero}
    \rho'(z^k;d^k)\le  \|F(z^k) + DF(z^k)d^k\| - \|F(z^k)\|
    \le (\eta_k-1)\|F(z^k)\|=(\eta_k-1)\rho(z^k)<0.
    \end{equation}
Assume to the contrary that for all $\ell\in \bN_0$, 
\[
\rho(z^k+\gam^\ell\bar \alpha_k d^k)> \rho(z^k)+\mu\bar \alpha_k\gam^\ell(\eta_k-1) \rho(z^k),
\]
that is,
\[
\frac{\rho(z^k+\gam^\ell\bar \alpha_k d^k)-\rho(z^k)}{\bar \alpha_k\gam^\ell}>\mu(\eta_k-1)\rho(z^k).
\]
Using \eqref{eq:phi ne zero} and passing to the limit as $\ell\to\infty$ we obtain (since $z^k\in \dom \rho$ which is open) that
\[
0>(\eta_k-1)\rho(z^k)\ge
\rho'(z^k; d^k)\geq  \mu(\eta_k-1)\rho(z^k),
\]
a contradiction as $\mu\in (0,1)$. Therefore, $\ell_k$ and $\alf_k$ are well defined. Since $z^0 \in \lev{\rho}{\beta_0} \subset \lev{\rho}{\beta}$, induction with the strict decrease just established gives $\rho(z^{k+1}) < \rho(z^k) < \beta$ for all $k$.
\end{proof}

By \cref{lem:alg well-defined}, each iterate of \cref{alg:inexact Newton} satisfies $\rho(z^{k+1}) < \rho(z^k)$, and therefore choosing any $\beta > \rho(z^0)$ ensures that all iterates remain in the sublevel set $\lev{\rho}{\beta}$. Compactness of this level set, established by \cref{lem:level cpt}, provides bounds that serve two purposes. First, the lower bound $\htaubetamin$ guarantees $\tau$ is strictly positive. \cref{alg:inexact Newton} uses this bound explicitly in \cref{step:provisional} to safeguard the Newton step from driving $\tau$ too close to zero. Second, the upper bounds on $\tau$ and $\norm{y}$ ensure that the Jacobian $DF$ remains invertible with bounded condition number throughout the iteration, so the Newton linear system in \cref{step:newton} is always solvable.

\begin{remark}[Fixed-scale specialization]\label{rem:fixed_scale}
When the scale $\btau>0$ is prescribed (e.g., the simplex-constrained case $\btau=1$, corresponding to KL-divergence minimization over $\Delta_n$), \cref{step:provisional} of \cref{alg:inexact Newton} is redundant and $\tau_k\equiv\btau$ in \cref{step:update_iterates}. The inexact Newton step~\eqref{eq:inexact_newton_step} reduces to the $m$-dimensional system for $\bar F(y):=b-\lam y-\btau Ap(y)$. Convergence follows from standard theory~\cite[Ch.~6, 8]{kelley1995iterative}, and \Cref{lem:DF inv bd,lem:step bd,lem:d bd,lem:Fdz residual,lem:beta Lipschitz bd,lem:alf bd} are not needed.
\end{remark}

\section{Bounds} \label{sec:bounds}

The iteration complexity and the asymptotic rates of convergence of Algorithm \ref{alg:inexact Newton} depend on a collection of bounds for $F$, $DF$, the search directions $d^k$, and the step-size $\alf_k$. These are developed individually in \cref{lem:step bd,lem:d bd,lem:Fdz residual,lem:beta Lipschitz bd,lem:alf bd}.

\paragraph{Standing setup}
Throughout this section, fix $\beta>0$, $\bar\eta\in[0,1)$, and $\eta\in[0,\bar\eta]$; let $\taubetamax$, $\taubetamin$, and $\ymax$ be as in \cref{lem:level cpt}. For $w\in\rr^m\times\rr_+$ and $\eta\in[0,1)$, define
\begin{equation}\label{eq:inexact Newton set}
\cN(w,\eta):=\set{s\in\rr^m\times \rr_+ | \norm{F(w) + DF(w)  s}\le\eta\norm{F(w)}}
\end{equation}
as the set of admissible inexact Newton directions at $w$ with relaxation parameter $\eta$.

\begin{lemma}[Bound on $\rho(z+s)$ along a perturbation]\label{lem:step bd}
For $z=(y,\tau)\in\lev{\rho}{\beta}$ and a generic perturbation
$s=(\Del y,\Del\tau)\in\rr^m\times\rr$ (not necessarily from $\cN$)
with $\norm{s}\le\del$ and $\tau+\Del\tau\ge\half\htaubetamin$,
\[
  \rho(z+\zeta s) \le \rhomax(\beta,\del) := \Fymax(\beta,\del)+\Ftaumax(\beta,\del)
  \quad\forall\,\zeta\in[0,1],
\]
where
\[\begin{aligned}
\Fymax(\beta,\delta)&:=
\taubetamax A_{\max}+\lam\ymax+\norm{b}+
(\lam+A_{\max})\del, \qquad\text{and}
\\
 \Ftaumax(\beta,\delta) &:= 1\!+\!\|c\|_\infty\!+\!A_{\max}(\ymax+\del)
 \\ &\qquad\qquad\ \ \ \  
\!+\!\max\{|\log(\half\htaubetamin)|,\taubetamax+\del\}\!+\!|\log \onorm{q}| .
\end{aligned}\]
\end{lemma}
\begin{proof}
Let $z=(y,\tau)\in\lev{\rho}{\beta}$ and let
$s=(\Del y,\Del \tau)\in\rr^m\times \rr$ be such that
$\norm{s}\le\del$ and $\tau+\Del\tau\ge\half\htaubetamin$.
For $\zeta\in[0,1]$ observe that 
\begin{equation}\label{eq:tau plus Del tau}
\tau+\zeta\Del\tau=(1-\zeta)\tau+\zeta(\tau+\Del\tau)
\ge (1-\zeta)\htaubetamin+\zeta \half\htaubetamin\ge \half\htaubetamin.
\end{equation}
Recall $p(y)\in\rint\Del_n$ from \cref{thm:primal_dual_duality}\cref{item:uniqueness}. Then
\[\begin{aligned}
F_y(z+\zeta s) &= b - (\tau+\zeta \Del \tau)Ap(y+\zeta \Del y) - \lambda(y+\zeta \Del y),
\\ &=
[\tau Ap(y+\zeta \Del y)-\lam y +b] -(\zeta Ap(y+\zeta \Del y))\Del \tau -(\lam\zeta)\Del y,
\\
F_\tau(z+\zeta s)&=-\logexp_{q/(\tau+\zeta \Del \tau)}(A\T (y+\zeta \Del y)-c)+1.
\end{aligned}\]
Therefore,  
\[
\begin{aligned}
\|F_y(z+\zeta s)\|&\le
\tau A_{\max}+\lam \norm{y}+\norm{b}+\zeta A_{\max}|\Del \tau|+\lam \zeta \norm{\Del y}
\\ &\le
\taubetamax A_{\max}+\lam\ymax+\norm{b}+ \zeta
\bip{\begin{pmatrix}\lam\\ A_{\max}\end{pmatrix}}{\begin{pmatrix}\norm{\Del y}\\ |\Del \tau|\end{pmatrix}}
\\ &\le
\taubetamax A_{\max}+\lam\ymax+\norm{b}+\zeta(\lam+A_{\max})\inorm{s}
\\ &\le
\taubetamax A_{\max}+\lam\ymax+\norm{b}+
(\lam+A_{\max})\del %
\equiv \Fymax(\beta,\del),
\end{aligned}
\]
where $\ymax$ is defined in \cref{lem:level cpt}. To bound $|F_\tau|$, define $h(y):=\logexp_q(A\T y-c)$. By \cref{lem:logexp_properties}\cref{item:logexp_composed}, the map $h$ is $A_{\max}$-Lipschitz. Hence,
\[\begin{aligned}
|F_\tau(z+\zeta s)|
&=
\left|-h(y+\zeta\Del y)+\log(\tau+\zeta\Del\tau)+1\right|
\\
&\le
1+\left|h(y+\zeta\Del y)-h(0)\right|+\left|h(0)\right|+\left|\log(\tau+\zeta\Del\tau)\right|
\\
&\le
1+A_{\max}\norm{y+\zeta\Del y}+\left|\logexp_q(-c)\right|+\left|\log(\tau+\zeta\Del\tau)\right|.
\end{aligned}\]
If $\tau+\zeta \Del \tau\ge1$, then 
$|\log(\tau+\zeta \Del \tau)|
=\log(\tau+\zeta \Del \tau)\le\tau + \zeta \Del \tau-1\le\tau+\zeta \Del \tau$, 
where the first inequality follows from the linearization of $\log$ at 1. 
If, on the other hand, $\tau+\zeta \Del \tau<1$, then 
$|\log(\tau+\zeta \Del \tau)|=-\log(\tau+\zeta \Del \tau)\le-\log(\half\htaubetamin)$ by monotonicity of the negative $\log$ and \eqref{eq:tau plus Del tau}. 
Combining these observations, we find 
$|\log(\tau+\zeta \Del \tau)|\le\max\{|\log(\half\htaubetamin)|,\taubetamax+\del\}$.
Therefore, 
\[\begin{aligned}
|F_\tau(z+\zeta s)|
&\le
1+A_{\max}(\ymax+\del)+\left|\logexp_q(-c)\right|
+\max\{|\log(\half\htaubetamin)|,\taubetamax+\del\} \\
&\le 1\!+\!\|c\|_\infty\!+\!A_{\max}(\ymax\!+\!\del)
\!+\!\max\{|\log(\half\htaubetamin)|,\taubetamax\!+\!\del\}\!+\!|\log \onorm{q}\! |,
\end{aligned}\]
and the RHS is equivalent to $\Ftaumax(\beta,\del)$.
Hence, $\rho(z+\zeta s)=\norm{F(z+\zeta s)}\le
\Fymax(\beta,\del)+\Ftaumax(\beta,\del)$.
\end{proof}

\begin{lemma}[Inexact Newton direction bounds]\label{lem:d bd}
For $z=(y,\tau)\in\lev{\rho}{\beta}$,
\[\begin{alignedat}{2}
  \tdmax(z,\eta)&:=\sup\rset{\norm{d}}{d\in\cN(z,\eta)}
  &\;\le\;& \max\{\lam^{-1},\tau\}(1+\eta)\rho(z),\ \text{and}
  \\
  \dmax(\beta,\eta)&:=
  \sup\rset{\tdmax(z,\eta)}{z\in\lev{\rho}{\beta}}
  &\;\le\;& \max\{\lam^{-1},\taubetamax\}(1+\eta)\beta.
\end{alignedat}\]
\end{lemma}
\begin{proof}
For $d\in\cN(z,\eta)$, we have
\[ \begin{aligned}
\norm{d}&=\norm{DF(z)^{-1}DF(z)d}
\\ &\le \norm{DF(z)^{-1}}\cdot\norm{F(z)+DF(z)d-F(z)}
\\ &\le \max\{\lam^{-1},\tau\}(1+\eta)\norm{F(z)}
\le
\max\{\lam^{-1},\taubetamax\}(1+\eta)\beta,
\end{aligned}
\]
where the second inequality uses Lemma \ref{lem:DF inv bd}\ref{item:singular-bounds}. 
\end{proof}

\begin{lemma}[Newton step residual bound]\label{lem:Fdz residual}
Let $\zeta\in [0,1]$, $z=(y,\tau)\in\lev{\rho}{\beta}$ and
$d=(\Del y,\Del \tau)\in \cN(z,\eta)$.
Set $\hat\del:=\dmax(\beta, \eta)$ and
\[
\bar\alpha \coloneqq
\begin{cases}
    1 &\text{if}\ \tau+\Del \tau\ge\half\htaubetamin,\\
   (\tau - \half\htaubetamin)/({-\Del \tau}) & \text{otherwise,}
\end{cases}
\]
Define the uniform lower bounds
\begin{equation}\label{eq:balf and hat alf}
\hat\alf\coloneqq
\begin{cases}
    1 &\text{if}\ \tau+\Del \tau\ge\half\htaubetamin,\\
   \htaubetamin/(2 \hdel) & \text{otherwise,}
\end{cases}
\qquad
\hat\eta \coloneqq 1-(1-\bar\eta)\hat\alf< 1.
\end{equation}
Then $\hat\alf\le\balf\le 1$, $\bar\eta\le\hat\eta$, and
\begin{align}
\label{eq:hat eta inexact Newton}
\norm{F(z)+\balf DF(z)d}
&\le[1-(1-\eta)\balf]\cdot\norm{F(z)}\le\hat\eta\norm{F(z)}\quad\text{and}
\\
\label{eq:updated residual bound}
  \norm{F(z+\zeta \bar\alpha d)}
  &\le \rhomax(\beta,\hat\del)\quad\forall\, \zeta\in[0,1].
\end{align}
\end{lemma}
\begin{proof}
Let $\zeta\in[0,1]$. If $\tau+\Del \tau\ge\half\htaubetamin$, then
$\bar\alf=1=\hat \alf$ and $\hat\eta=\bar\eta$. In this case,
\eqref{eq:hat eta inexact Newton} follows from  $d\in\cN(z,\eta)$
while \eqref{eq:updated residual bound}
follow from \cref{lem:step bd}.
If $\tau+\Del\tau<\half \htaubetamin$, then $\Del\tau<0$ since 
$\tau\ge\taubetamin>\half \htaubetamin$. Therefore, 
$0<\hat\alf=(\tau-\half \htaubetamin)/(-\Del\tau)<1$. Moreover,
$0<\hat\alf=\htaubetamin/(2\hdel)<(\tau-\half \htaubetamin)/\hdel
\le (\tau-\half \htaubetamin)/(-\Del\tau)=\bar\alf<1$
since $(-\Del\tau)=|\Del\tau|\le \norm{d}$. Consequently, 
$0<\bar\eta<\hat\eta$ and 
\begin{equation*}\label{eq:rho cvx comp}
\begin{aligned}
\norm{F(z)+\balf DF(z)d}
&\le (1-\balf)\norm{F(z)}+\balf \norm{F(z)+DF(z)d}
\\ 
& \le (1-\balf + \balf\eta)\norm{F(z)}
\\
&= [1-(1-\eta)\balf]\norm{F(z)}
\\ & \le\hat\eta\norm{F(z)},
\end{aligned}
\end{equation*}
that is, \eqref{eq:hat eta inexact Newton} holds.
The inequality \eqref{eq:updated residual bound} also holds
since $\zeta\balf\in[0,1)$.
\end{proof}

\begin{lemma}[Lipschitz bound for $DF$]\label{lem:beta Lipschitz bd}
Let $z = (y, \tau) \in \rr^m \times [\tau_l, \tau_u]$
with $0 < \tau_l \le \tau_u$.
The Jacobian $DF$ is Lipschitz continuous on $ \rr^m \times [\tau_l, \tau_u]$
 with Lipschitz constant $\tL(\tau_l, \tau_u)$ satisfying %
\begin{equation}\label{eq:lipschitz_constant}
\tL(\tau_l, \tau_u)\le 3\tau_u\norm{A}^3 + \tfrac{3}{2}\norm{A}^2 + 1/\tau_l^2.
\end{equation}
In particular, this implies that, for $\hat\beta>0$, $DF$ is Lipschitz continuous on
the convex set
$\rr^m \times [\half \tau_{\hat\beta{min}}, 2\tau_{\hat\beta{max}}]$
with
a Lipschitz constant $\LD(\hat\beta)$ satisfying %
\begin{equation}\label{eq:beta Lipschitz constant}
\LD(\hat\beta)\le 6 \tau_{\hat\beta{max}}\norm{A}^3 + \tfrac{3}{2}\norm{A}^2 +
 4/\tau_{\hat\beta{min}}^2,
\end{equation}
where
$\lev{\rho}{\hat\beta}\subset\interior(\rr^m \times [\half \tau_{\hat\beta{min}}, 2\tau_{\hat\beta{max}}])$.
\end{lemma}
\begin{proof}
We use the notation for $DF$ developed in the proof of \cref{lem:DF inv bd}.
Let $z^i = (y^i, \tau_i)$ for $i = 1, 2$, and write
$\Delta y = y^1 - y^2$, $\Delta\tau = \tau_1 - \tau_2$, and 
$p^i = \dlogexp(A\T y^i - c)$. We bound the Lipschitz constants for each block of $DF$ as defined in~\eqref{eq:dual_hess}.

\noindent
\emph{Lipschitz bound on $S(y)$.}
\Cref{lem:logexp_properties}\ref{item:logexp_smooth} gives $\dlogexp$ is (1/2)-Lipschitz.
Composing with $A\T$ gives $\norm{p^1 - p^2} \le \norm{A}\norm{\Delta y}$. Expanding 
$S(y^1) - S(y^2) = \Diag{p^1 - p^2} - p^1(p^1 - p^2)\T  - (p^1 - p^2)(p^2)\T $ and applying the triangle inequality with $\norm{p_i} \le 1$ yields
\begin{equation}\label{eq:S_lipschitz}
\norm{S(y^1) - S(y^2)} \le 3\norm{p^1 - p^2} \le 3\norm{A}\norm{\Delta y}.
\end{equation}

\noindent
\emph{Block $H$ (top-left, $m \times m$).}
Adding and subtracting $\tau_1 A S(y^2) A\T $ gives
\[
\Delta H := H_1 - H_2 = -\tau_1 A(S(y^1) - S(y^2))A\T  - (\tau_1 - \tau_2)A S(y^2)A\T .
\]
Applying \eqref{eq:S_lipschitz} and Lemma
\ref{lem:logexp_properties}\ref{item:logexp_smooth}:
\begin{equation}\label{eq:H_lipschitz}
\norm{\Delta H} \le 3\tau_{u}\norm{A}^3\norm{\Delta y} + \tfrac{1}{2}\norm{A}^2|\Delta\tau|.
\end{equation}

\noindent
\emph{Block $g$ (off-diagonal, $m \times 1$).}
Since $g = -Ap$, 
we have
\begin{equation*}\label{eq:g_lipschitz}
\norm{\Delta g} = \norm{A(p^1 - p^2)} \le \norm{A}^2\norm{\Delta y},
\end{equation*}
where $\Delta g := g^1 - g^2$.

\noindent
\emph{Block $1/\tau$ (bottom-right, scalar).}
Direct computation yields
\begin{equation}\label{eq:tau_lipschitz}
\left|\frac{1}{\tau_1} - \frac{1}{\tau_2}\right| = \frac{|\Delta\tau|}{\tau_1\tau_2} \le \frac{|\Delta\tau|}{\tau_{l}^2}.
\end{equation}

\noindent
\emph{Assembly.}
Using the decomposition $DF = M_1 + M_2$ from \eqref{eq:DF decomposition},
write $DF(z_1) - DF(z_2) = \Delta M_1 + \Delta M_2$ where $\Delta M_1$ is block-diagonal with blocks $\Delta H$ and $\Delta(1/\tau)$, and $\Delta M_2$ is the off-diagonal part with $\norm{\Delta M_2} = \norm{\Delta g}$. By the triangle inequality and~\eqref{eq:H_lipschitz}--\eqref{eq:tau_lipschitz}:
\[
\norm{DF(z_1) - DF(z_2)} \le 
{(3\tau_{u}\norm{A}^3 + \norm{A}^2)}\cdot\norm{\Delta y} + 
{\bigl(\tfrac{1}{2}\norm{A}^2 + 1/\tau_{l}^2\bigr)}|\Delta\tau|.
\]
This yields~\eqref{eq:lipschitz_constant} via the Cauchy--Schwarz inequality and the
inequality $\norm{z}\le\onorm{z}$.
\end{proof}

\begin{lemma}[Step-size bound]\label{lem:alf bd}
Let $\hM(\beta):=\mathM(2\taubetamax)$ where $\mathM(\cdot)$ is as in \cref{lem:DF inv bd}.
If $\rho(z)>0$ and
$\tilde \alf\in (0,1]$ is such that
$z+\tilde\alf d\in \lev{\rho}{\hat\beta}$, with $\hbeta>\beta$, and
\begin{equation}\label{eq:alf bbd below1}
\rho(z)+\mu \tilde\alf (\eta-1)\rho(z)<\rho(z+\tilde\alf d),
\end{equation}
then
\begin{equation}\label{eq:alf bound sharp}
1\ge\tilde\alf>\frac{2(1-\mu)(1-\eta)}{\hLD(\hat\beta) \hM(\beta)^2(1+\eta)^2\norm{F(z)}}.
\end{equation}
Since $\eta\le\bar\eta$ implies $2/(1+\eta)^2\ge 1/2$, and $\norm{F(z)}\le\beta$,
\begin{equation}\label{eq:alf bound}
\tilde\alf>
\frac{(1-\mu)(1-\bar\eta)}{2\hLD(\hat\beta) \hM(\beta)^2\beta} > \alf^*,
\quad\text{where}\quad
\alf^*:=(1-\mu)(1-\bar\eta)\bigl(1+2\hLD(\hat\beta)\hM(\beta)^2\beta\bigr)^{-1}.
\end{equation}
\end{lemma}
\begin{proof}
The convexity of the norm and definition of $d$ tells us that
\begin{equation}\label{eq:cvx norm}
\norm{F(z)+\talf DF(z)d}-\norm{F(z)}\le 
\talf\bigl(\norm{F(z)+ DF(z)d}-\norm{F(z)}\bigl)
\le \talf(\eta-1)\norm{F(z)}.
\end{equation}
By \cref{lem:d bd,lem:beta Lipschitz bd},
\eqref{eq:cvx norm}, and the quadratic bound lemma \cite[Theorem 3.2.12]{ortega1970itersols},
we have
\[\begin{aligned}
\norm{F(z+\talf d)}-\norm{F(z)}
&=
\left(\norm{F(z)+\talf DF(z)d}-\norm{F(z)}\right)
\\ &\qquad\qquad +
\left(\norm{F(z+\talf d)}-\norm{F(z)+\talf DF(z)d}\right)
\\ &\le
\talf(\eta-1)\norm{F(z)}+\norm{F(z+\talf d)-[F(z)+\talf DF(z)d]}
\\ &\le
\talf(\eta-1)\norm{F(z)}+\half\hLD(\hat\beta){\talf}^2\norm{d}^2
\\ &\le 
\talf(\eta-1)\norm{F(z)}+
\half\hLD(\hat\beta){\talf}^2\hM(\beta)^2(1+\eta)^2\norm{F(z)}^2.
\end{aligned}\]
Hence, by~\eqref{eq:alf bbd below1},
\[\begin{aligned}
\mu \talf (\eta-1)\rho(z)&<\rho(z+\talf d)-\rho(z)
 \\ &\le
 \talf(\eta-1)\rho(z)+\half\hLD(\hat\beta){\talf}^2\hM(\beta)^2(1+\eta)^2\rho(z)^2.
\end{aligned}\]
Dividing by $\talf\rho(z)$ tells us that
\[
\mu(\eta-1)<(\eta -1)+\half\hLD(\hat\beta){\talf}\hM(\beta)^2(1+\eta)^2\rho(z).
\]
By rearranging terms, we obtain
\[
1\ge\talf>\frac{2(1-\mu)(1-\eta)}{\hLD(\hat\beta) \hM(\beta)^2(1+\eta)^2\rho(z)}
 \ge
 \frac{(1-\mu)(1-\bar\eta)}{2\hLD(\hat\beta) \hM(\beta)^2\rho(z)}
 >
 \frac{(1-\mu)(1-\bar\eta)}{2\hLD(\hat\beta) \hM(\beta)^2\beta}
\]
since $2/(1+\bar \eta)^2>\half$. Hence \eqref{eq:alf bound sharp} and \eqref{eq:alf bound} follow.
\end{proof}

\section{Convergence} \label{sec:convergence}

This section establishes the main convergence result for \cref{alg:inexact Newton}. \Cref{thm:convergence of inexact Newton 1} describes (1) the global linear convergence of the merit $\rho$ at an explicit rate $\hat\nu<1$; (2) the iteration complexity needed to reduce $\rho(z^k)$ or $\norm{z^{k+1}-z^*}$ below $\eps$; and (3) superlinear (resp.\ $(1+p)$-superlinear) local rates under standard forcing-sequence conditions on $\{\eta_k\}$. The argument relies on the master bounds of \cref{lem:step bd,lem:d bd,lem:Fdz residual,lem:beta Lipschitz bd,lem:alf bd} and the Jacobian bounds of \cref{lem:DF inv bd}. Together these control the step size through the three-case analysis that follows.

\begin{theorem}[Iteration complexity and asymptotic convergence rates]
\label{thm:convergence of inexact Newton 1}
Let $x^*\in \rr^n_{++}$ and $z^*:=(y^*,\tau^*)\in \rr^m\times\rr_{++}$
be the unique primal-dual solutions to the primal-dual pair of problems given in \cref{thm:primal_dual_duality}. Let $\{z^k=(y^k,\tau_k)\}_{k\ge 0}$ be the sequence generated by Algorithm
\ref{alg:inexact Newton} with initialization $z^0 \in\rr^m\times\rr_{++}$, $\eps\ge0$, $\mu,\gam\in(0,1)$, $\bar\eta\in[0,1)$, $\{\eta_k\}_{k\ge 0}\subset [0,\bar \eta]$, and $\beta>\rho(z^0)=\beta_0$. Define the corresponding primal sequence
by $x^k := \tau_k p(y^k)$.
Then
\begin{equation}\label{eq:linear rate}
\rho(z^{k+1})\le \hat\nu\,\rho(z^k)\qquad\forall\, k\ge 0,
\end{equation}
where $\hat\nu:=1-\mu{\hat\alf}^*(1-{\hat\eta})\in(0,1)$
with ${\hat\alf}^*=\min\{\hat\alf,\gam\alf^*\}$, and $\hat\eta$, $\hat\alf$, $\alf^*$ are given in~\eqref{eq:balf and hat alf} and~\eqref{eq:alf bound}.
Moreover, the following convergence properties hold.
\begin{thmenum}[label=(\arabic*),ref=(\arabic*)]
\item \label{item:complexity}
Iteration Complexity:
If $\eps\in (0,\beta_0)$, then
\[
\rho(z^k)\le\eps\quad  \text{for all}\quad k\ge \log(\beta_0\eps^{-1})/\log(\hat\nu^{-1})=O(\log(\eps^{-1})),
\]
and
\[
\norm{z^{k+1}-z^*}\le\eps\quad  \text{for all}\quad
k\ge \log(K\eps^{-1})/\log(\hat\nu^{-1})=O(\log(\eps^{-1})),
\]
where $K:=\max\{\lam^{-1},\, \taubetamax\}(1+\bar\eta)\beta_0(1-\hat\nu)^{-1}$.
\item \label{item:asymptotic}
Asymptotic rates: Let $\eps=0$. Then $z^k\rightarrow z^*$, $\alf_k=1$ for all $k$ 
sufficiently 
large, and we have the following
asymptotic rates of convergence.
 \begin{thmenum}[label=(\roman*),ref=(\roman*)]
        \item \label{item:super}
        (Superlinear)  If $\eta_k \to 0$, then
        \[
        \lim_{k \to \infty} \frac{\|z^{k+1} - z^\ast \|}{\|z^{k} - z^\ast \|} = 0.
        \] 
        \item \label{item:super-p}
        ((1+p)-Superlinear) If eventually $\eta_k = O(\|F(z^k)\|^p)$, then 
        for all $k$ sufficiently large, 
        \[
        \|z^{k+1} - z^\ast \|  = O(\|z^{k} - z^\ast \|^{1+p}).
        \]
    \end{thmenum}
\item \label{item:primal convergence}
(Primal iterates) The primal iterates $x^k$ satisfy
\[
    \norm{x^*-x^k}\leq (\half\tau^*\norm{A}+1)\norm{z^*-z^k};
\]
hence $\{x^k\}$ inherits the iteration complexity and asymptotic rates of $\{z^k\}$.
\end{thmenum}
\end{theorem}

\begin{proof}
The proof has three parts corresponding to items (1)--(3). Part~(1) establishes the per-iteration contraction~\eqref{eq:linear rate} and derives the iteration complexity by geometric summation together with \cref{lem:d bd}. Part~(2) shows that, once the iterates enter a neighborhood of~$z^*$, the step size stabilizes at $\alf_k=1$, after which the inexact Newton identity and Lipschitz continuity of~$DF$ yield the claimed asymptotic rates. Part~(3) transfers the dual rates to the primal iterates.

Throughout the proof, set $\hat\beta:=\rhomax(\beta,\dmax(\beta,\bar\eta))$; the constants $\hat\alf$ and $\hat\eta$ are defined in~\eqref{eq:balf and hat alf}, and $\alf^*$ in~\eqref{eq:alf bound}.

By Lemma \ref{lem:alg well-defined}, the algorithm is well-defined and
can only terminate if $\rho(z^{k+1})\le\eps$.
By Step 3 and \cref{lem:Fdz residual},
\begin{equation*}
0<\hat\alf\le\balf_k \coloneqq
\begin{cases}
    1 & \text{if $\tau_k+\Del\tau_k\ge\half\htaubetamin$,}\\
    \frac{\tau_k - \half\htaubetamin}{-\Del\tau_k} & \text{otherwise,}
\end{cases}
\qquad\forall\, k\ge 0.
\end{equation*}
By \cref{lem:alg well-defined,lem:Fdz residual},
$\{z^k+\balf_kd^k\}\subset\lev{\rho}{\hat\beta}$ since
$\{z^k\}\subset\lev{\rho}{\beta}$, and
\[
\norm{F(z^k)+\balf_k DF(z^k)d^k}\le\hat\eta\norm{F(z^k)}\quad\forall\, k\ge 0.
\]
Let $\ell_k$ be such that $\alf_k=\balf_k\gam^{\ell_k}$.
We now consider three possible cases for $\alf_k$: 
\[
(a)\ \alf_k=1,\quad (b)\ \alf_k<1\quad \text{and}\quad \ell_k=0, \quad\text{or}\quad
(c)\ \ell_k\ge1.
\]

\noindent
\emph{Case (a).}
If $\alf_k=1$, then $\ell_k=0$ and $\balf_k=1$, so \cref{step:backtrack} and $\eta_k\le\bar\eta\le\hat\eta<1$, together with $\hat\alf\le 1$, give
\[
\rho(z^{k+1})\le (1-\mu(1-\eta_k))\rho(z^k)
\le (1-\mu\hat\alf(1-\hat\eta))\rho(z^k).
\]

\noindent
\emph{Case (b).}
If $\alf_k<1$ and $\ell_k=0$, then
\[
\alf_k=\balf_k=\frac{\tau_k-\half\htaubetamin}{-\Del\tau_k}\ge \hat\alf>0,
\]
so
\[
\rho(z^{k+1})\le (1-\mu\balf_k(1-\eta_k))\rho(z^k)
\le (1-\mu\hat\alf(1-\hat\eta))\rho(z^k).
\]

\noindent
\emph{Case (c).}
If $\ell_k\ge 1$, then \cref{step:backtrack} gives
\[
\rho(z^k)+\mu\gam^{-1}\alf_k(\eta_k-1)\rho(z^k)
<\rho(z^k+\gam^{-1}\alf_k d^k)
\]
with $\alf_k=\gam^{\ell_k}\balf_k$, $\gam^{\ell_k-1}\in[0,1]$ and 
\[
z^k+\gam^{\ell_k-1}\balf_k d^k\in 
\rset{z+\zeta s}{
\begin{matrix}
z=(y,\tau)\in\lev{\rho}{\beta} 
\\
\zeta\in[0,1],\, s=(\Del y,\Del \tau)\in\rr^m\times \rr,
\\ 
\norm{s}\le\dmax(\beta,\bar\eta) \, \text{and}\ \tau+\Del\tau\ge\half\htaubetamin
\end{matrix}
}.
\]
Since $\ell_k\ge 1$ implies $\gam^{\ell_k-1}\in[0,1]$, \cref{lem:step bd} gives
$\rho(z^k+\gam^{\ell_k-1}\balf_k d^k)\le \hat\beta$, i.e.,
$z^k+\gam^{-1}\alf_k d^k\in \lev{\rho}{\hat \beta}$.
Thus, \cref{lem:alf bd} gives
$1>\gam^{-1}\alf_k>\alf^*>0$, so that
\[
\rho(z^{k+1})\le (1-\mu\alf_k(1-\eta_k))\rho(z^k)\le (1-\mu\gam\alf^*(1-\hat\eta))\rho(z^k).
\]
Combining the three cases with
\(
{\hat\alf}^*:=\min\{\hat\alf,\gam\alf^*\}
\) and $\hat\nu:=1-\mu{\hat\alf}^*(1-{\hat\eta})\in(0,1)$ yields
\[
\rho(z^{k+1})\le \hat\nu\,\rho(z^k)\qquad\forall\, k\ge 0,
\]
which proves~\eqref{eq:linear rate}.

\medskip

Part~\ref{item:complexity} Iteration Complexity.
Iterating~\eqref{eq:linear rate} with $z^0\in\lev{\rho}{\beta_0}$ gives
$\rho(z^k)\le \hat\nu^k\beta_0$ for all $k\ge 0$. Thus $\rho(z^k)\le\eps$ when
\[k\ge \log(\beta_0\eps^{-1})/\log(\hat\nu^{-1})=O(\log(\eps^{-1})).\]
Since $\{z^k\}\subset\lev{\rho}{\beta}$ is compact by \cref{lem:level cpt} and $\rho(z^k)\to 0$, the only cluster point of $\{z^k\}$ is the unique zero $z^*$ of~$\rho$, so $z^k\to z^*$.

Next, by \cref{lem:d bd}, we have
\[\begin{aligned}
\norm{z^{k+1}-z^*}&\le\sum_{j=k}^\infty \norm{z^{j+1}-z^{j}} = \sum_{j=k}^\infty\alf_j\norm{d^j} \le \max\{\lam^{-1},\, \taubetamax\}(1+\bar\eta)\sum_{j=k}^\infty \rho(z^j)
\\ &\le
\max\{\lam^{-1},\, \taubetamax\}(1+\bar\eta)\sum_{j=k}^\infty \hat\nu^j\beta_0 = 
\max\{\lam^{-1},\, \taubetamax\}(1+\bar\eta)\beta_0\frac{\hat\nu^k}{1-\hat\nu}.
\end{aligned}\]
Therefore, $\norm{z^{k+1}-z^*}\le\eps$ when
\[
k\ge \log(K\eps^{-1})/\log(\hat\nu^{-1})=O(\log(\eps^{-1})),
\]
where $K:=\max\{\lam^{-1},\, \taubetamax\}(1+\bar\eta)\beta_0(1-\hat\nu)^{-1}$.

Part~\ref{item:asymptotic} Asymptotic rates:
Assume $\eps=0$. If the algorithm terminates finitely then $z^k=z^*$ and the claim holds trivially, so assume $\rho(z^k)>0$ for all~$k$. Part~\ref{item:complexity} gives $z^k\to z^*$. The argument proceeds in three steps: (a)~the step size stabilizes at $\alf_k=1$, (b)~an exact error identity holds, and (c)~the advertised rates follow.

\emph{Step (a): Step-size stabilization.}
Since $z^k\to z^*$ and $\tau^*\in[\taubetamin,\taubetamax]$ by \cref{lem:level cpt}, there exist $\delta>0$ and $k_0\ge 1$ such that
\[
\{z^k\}_{k\ge k_0}\subset\bB(z^*;\delta)\subset \rr^m\times[\half\taubetamin,2\taubetamax].
\]
By \cref{lem:d bd}, $d^k\to 0$, so $k_0$ can be chosen so that $\{z^k+d^k\}_{k\ge k_0}\subset\bB(z^*;\delta)$ as well, whence $\balf_k=1$ for all $k\ge k_0$.
Let $\cL(\beta)$ and $\mathM(2\taubetamax)$ be as in \cref{lem:DF inv bd}\ref{item:level-set-bounds}, and let $\LD(\hat\beta)$ be as in~\eqref{eq:beta Lipschitz constant}. Since $F(z^k)\to 0$, the lower bound in~\eqref{eq:alf bound sharp} satisfies
\[
\frac{(1-\mu)(1-\bar\eta)}{2\LD(\hat\beta)\,\mathM(2\taubetamax)^2\norm{F(z^k)}}\to +\infty,
\]
so \eqref{eq:alf bbd below1}--\eqref{eq:alf bound sharp} give $\rho(z^k+d^k)\le\rho(z^k)+\mu(\eta_k-1)\rho(z^k)$ for all $k$ sufficiently large. Hence $\alf_k=1$ for all $k\ge k_0$.

\emph{Step (b): Error identity.}
For $k\ge k_0$, using $F(z^*)=0$,
\[\begin{aligned}
z^{k+1}-z^* &= z^k-z^*+d^k
\\ &= DF(z^k)^{-1}\bigl[\,F(z^*)-(F(z^k)+DF(z^k)(z^*-z^k))\bigr]
\\ &\quad +DF(z^k)^{-1}\bigl(F(z^k)+DF(z^k)d^k\bigr).
\end{aligned}\]

\emph{Step (c): Rate bound.}
Taking norms,
\[\begin{aligned}
\norm{z^{k+1}-z^*}
&\le \mathM(2\taubetamax)\left[\half\LD(\hat\beta)\norm{z^k-z^*}^2 + \eta_k\norm{F(z^k)-F(z^*)}\right]
\\ &\le \mathM(2\taubetamax)\left[\half\LD(\hat\beta)\norm{z^k-z^*}^2 + \eta_k\cL(\beta)\norm{z^k-z^*}\right].
\end{aligned}\]
The first inequality uses the quadratic bound lemma~\cite[Theorem~3.2.12]{ortega1970itersols} with $DF$ being $\LD(\hat\beta)$-Lipschitz on $\rr^m\times[\half\tau_{\hat\beta,\min},2\tau_{\hat\beta,\max}]$ by \cref{lem:beta Lipschitz bd} together with the inexact Newton condition in Step~\ref{step:newton} of \cref{alg:inexact Newton}. The second inequality uses the $\cL(\beta)$-Lipschitzness of~$F$ on $\rr^m\times[\half\taubetamin,2\taubetamax]$ from \cref{lem:DF inv bd}, which applies because the line segment from~$z^k$ to~$z^*$ lies in the $\beta$-strip by Step~(a).
If $\eta_k\to 0$, then the ratio $\norm{z^{k+1}-z^*}/\norm{z^k-z^*}\to 0$, proving~\ref{item:super}; if $\eta_k=O(\norm{F(z^k)}^p)$ and $\norm{F(z^k)}=O(\norm{z^k-z^*})$ (from Lipschitzness), then $\norm{z^{k+1}-z^*}=O(\norm{z^k-z^*}^{1+p})$, proving~\ref{item:super-p}.

Part~\ref{item:primal convergence} Primal iterates.
The triangle inequality applied to $x^*-x^k = \tau^*\bigl(p(y^*)-p(y^k)\bigr) + (\tau^*-\tau^k)p(y^k)$, together with \cref{lem:logexp_properties}\cref{item:logexp_composed} and $\norm{p(y^k)}\le 1$, yields
\[
  \norm{x^*-x^k}
  \le \half\tau^*\norm{A}\norm{y^*-y^k} + |\tau^*-\tau^k|
  \le \bigl(\half\tau^*\norm{A}+1\bigr)\norm{z^*-z^k}.
\]
The primal iterates thus inherit the rates of Parts~\ref{item:complexity}--\ref{item:asymptotic} with constant $\half\tau^*\norm{A}+1$.
\end{proof}

\section{Perturbation analysis}\label{sec:perturbation}

With the prior parameterized by $r:=\vlog q$, the solution map $(b,\lam,r)\mapsto (y^\ast,\tau^\ast,x^\ast)(b,\lam,r)$ is jointly $\cC^\infty$ on $\rr^m\times\rr_{++}\times\rr^n$ (\cref{lem:joint_perturbation_derivatives}), locally Lipschitz with computable constants (\cref{thm:joint_perturbation}), and continuous as $\lam\downarrow 0$ (\cref{thm:lam convergence}).

\paragraph{Notation} Write $z^\ast=(y^\ast,\tau^\ast)$, $p^\ast:=p(y^\ast)$, $S^\ast:=S(y^\ast)$, and $\tH^\ast:=\tH(z^\ast)=\lam I+\tau^\ast A\Diag{p^\ast}A\T$ (cf.~\cref{lem:DF inv bd}\cref{item:DF inverse}); each of $y^\ast$, $\tau^\ast$, $p^\ast$, $S^\ast$, $\tH^\ast$ is treated throughout this section as a function of $(b,\lam,r)$. Let $D_\theta$ denote the Jacobian with respect to a parameter $\theta\in\set{b,\lam,r}$ and $D_{(y,\tau)}$ the Jacobian with respect to $(y,\tau)$.

\begin{lemma}[Jacobians of the solution map]\label{lem:joint_perturbation_derivatives}
The solution map $(b,\lambda,r)\mapsto x^\ast$ is in $\mathcal{C}^\infty(\rr^m\times\rr_{++}\times\rr^n)$ with partial Jacobians:
\begin{equation}\label{eq:Dbxstar}
D_bx^\ast=\tau^\ast\Diag{p^\ast}A^\top (\tH^\ast)^{-1},
\qquad
D_\lambda x^\ast=-(D_bx^\ast)y^\ast,
\end{equation}
and
\begin{equation}\label{eq:Drxstar}
D_rx^\ast
=
\tau^\ast\Diag{p^\ast}
-(\tau^\ast)^2\Diag{p^\ast}A^\top (\tH^\ast)^{-1}A\Diag{p^\ast}.
\end{equation}
\end{lemma}

\begin{proof}
Let $y^\ast := y^\ast(b,\lambda,r)$, $\tau^\ast := \tau^\ast(b,\lambda,r)$, and $x^\ast := x^\ast(b,\lambda,r)$ for brevity, where $r = \vlog q$.
Consider the map $(y,\tau,b,\lambda,r)\mapsto F(y,\tau,b,\lambda,r)$ where $F$ is as in~\eqref{eq:dual_grad}, now incorporating the additional inputs $(b,\lambda,r)$ through $q=\exp(r)$:
\begin{equation}\label{eq:F-more-params}
 F(y,\tau,b,\lambda,r) =  \begin{bmatrix}
     b - \lambda y - \tau A p(y)  \\[3pt]
    -\log \ip{\ones, \vexp( A\T y - c + r)} + \log \tau + 1
   \end{bmatrix}.
\end{equation}
By \cref{lem:logexp_properties}\cref{item:logexp_smooth} and the definition \eqref{eq:F-more-params},
\[
F\in \mathcal{C}^\infty(\rr^m\times\rr_{++}\times\rr^m\times\rr_{++}\times\rr^n),
\]
and by \cref{thm:primal_dual_duality}\cref{item:DF-invertible}, for all $(b,\lambda,r)\in\rr^m\times\rr_{++}\times\rr^n$,
$D_{(y,\tau)}F(y,\tau,b,\lambda,r)$ is non-singular for every $(y,\tau)\in\rr^m\times\rr_{++}$.
Also, $F(y^\ast(b,\lambda,r),\tau^\ast(b,\lambda,r),b,\lambda,r)=0$ by
\cref{thm:primal_dual_duality}\cref{item:uniqueness}. Hence, by the Implicit Function Theorem
\cite[Theorem~1B.1]{DontRock}, the map $(b,\lambda,r)\mapsto (y^\ast(b,\lambda,r),\tau^\ast(b,\lambda,r))$ is $\mathcal{C}^\infty$, and thus so is $(b,\lambda,r)\mapsto x^\ast(b,\lambda,r)$ on its domain since $x^\ast=\tau^\ast p(y^\ast)$. Moreover, \cite[Theorem~1B.1]{DontRock} provides
\begin{equation}\label{eq:IFT_linear_system}
D_{(y,\tau)}F(y^\ast,\tau^\ast,b,\lambda,r)
\begin{bmatrix}
D_b y^\ast & D_\lambda y^\ast & D_r y^\ast \\
D_b \tau^\ast & D_\lambda \tau^\ast & D_r \tau^\ast
\end{bmatrix}
=
- D_{(b,\lambda,r)}F(y^\ast,\tau^\ast,b,\lambda,r),
\end{equation}
where $D_b y^\ast\in\rr^{m\times m}$, $D_\lambda y^\ast\in\rr^{m\times 1}$, $D_r y^\ast\in\rr^{m\times n}$,
$D_b\tau^\ast\in\rr^{1\times m}$, $D_\lambda\tau^\ast\in\rr$, and $D_r\tau^\ast\in\rr^{1\times n}$.
From \eqref{eq:dual_grad} and \eqref{eq:dual_hess}, we obtain
\[
D_{(b,\lambda,r)}F(y^\ast,\tau^\ast,b,\lambda,r)
=
\begin{bmatrix}
I & -y^\ast & -\tau^\ast A S(y^\ast) \\
0 & 0 & -p(y^\ast)^\top
\end{bmatrix}
\]
and
\[
D_{(y,\tau)}F(y^\ast,\tau^\ast,b,\lambda,r)
=
\begin{bmatrix}
-\lambda I-\tau^\ast A S(y^\ast)A^\top & -Ap(y^\ast) \\
-(Ap(y^\ast))^\top & 1/\tau^\ast
\end{bmatrix}.
\]
Solving the linear system \eqref{eq:IFT_linear_system} via \Cref{lem:DF inv bd} yields 
\[\begin{aligned}
D_b y^\ast &= (\tH^*)^{-1}&, \quad D_b \tau^\ast 
&= \tau^\ast (Ap^\ast)^\top (\tH^*)^{-1}, \\
D_\lambda y^\ast &= -(\tH^*)^{-1}y^\ast&, \quad 
D_\lambda \tau^\ast &= -\tau^\ast (Ap^\ast)^\top (\tH^*)^{-1}y^\ast,
\end{aligned}\]
and
\[
D_r y^\ast \!=\! -\tau^\ast (\tH^*)^{-1}\!A\Diag{p^\ast},
\quad
D_r \tau^\ast \!=\! \tau^\ast {p^\ast}^\top \!\!- \!(\tau^\ast)^2(Ap^\ast)^\top \!(\tH^*)^{-1}\!A\Diag{p^\ast}.
\]

Applying the product and chain rule to $x^\ast=\tau^\ast p^\ast$, we obtain
\[\begin{aligned}
D_bx^\ast
&=
p^\ast D_b\tau^\ast+\tau^\ast (D_yp^\ast)(D_by^\ast)
=
p^\ast \tau^\ast (Ap^\ast)^\top (\tH^*)^{-1}+\tau^\ast S^\ast A^\top (\tH^*)^{-1}
\\ &=
\tau^\ast \Diag{p^\ast}A^\top (\tH^*)^{-1},
\end{aligned}\]
and
\[\begin{aligned}
D_\lambda x^\ast
&\!=\!
p^\ast D_\lambda\tau^\ast\!+\!\tau^\ast (D_yp^\ast)(D_\lambda y^\ast)
\!=\!
-p^\ast \tau^\ast (Ap^\ast)^\top \!(\tH^*)^{-1}y^\ast\!-\!\tau^\ast S^\ast A^\top \!(\tH^*)^{-1}y^\ast
\\ &\!=\!
-(D_bx^\ast)y^\ast.
\end{aligned}\]
which establishes \eqref{eq:Dbxstar}. Moreover, noting that $D_r p^\ast = S^\ast\bigl(A^\top D_r y^\ast + I\bigr)$, we have
\[
\begin{aligned}
&D_r x^\ast
= p^\ast D_r\tau^\ast+\tau^\ast D_r p^\ast \\
&=
p^\ast\!\left(\!\tau^\ast {p^\ast}^\top\!\!\!-\!(\tau^\ast)^2(Ap^\ast)^\top \!(\tH^*)^{-1}\!\!A\Diag{p^\ast}\!\right)
\!+\!\tau^\ast S^\ast\!\!\left(\!-\tau^\ast \!A^\top (\tH^*)^{-1}\!\!A\Diag{p^\ast}\!+\!I\!\right) \\
&=
\tau^\ast \bigl(p^\ast {p^\ast}^\top+S^\ast\bigr)
-(\tau^\ast)^2\bigl(p^\ast {p^\ast}^\top+S^\ast\bigr)A^\top (\tH^*)^{-1}A\Diag{p^\ast} \\
&=
\tau^\ast\Diag{p^\ast}
-(\tau^\ast)^2\Diag{p^\ast}A^\top (\tH^*)^{-1}A\Diag{p^\ast},
\end{aligned}
\]
which yields \eqref{eq:Drxstar}. 
\end{proof}

The partial Jacobians of \cref{lem:joint_perturbation_derivatives} integrate along a linear path in parameter space to give the joint Lipschitz bound below.

\begin{theorem}[Joint perturbation sensitivity]\label{thm:joint_perturbation}
For $(b_i,\lambda_i,r_i)\in\rr^m\times\rr_{++}\times\rr^n$, $i=1,2$, define the linear paths $b(t):=(1-t)b_1+tb_2$, $\lambda(t):=(1-t)\lambda_1+t\lambda_2$, and $r(t):=(1-t)r_1+tr_2$, and let $\lambda_-:=\min\{\lambda_1,\lambda_2\}$ and
\[
\bar\tau:=\max_{t\in[0,1]}\tau^\ast(b(t),\lambda(t),r(t)), \quad
\bar y :=\max_{t\in[0,1]}\|y^\ast(b(t),\lambda(t),r(t))\|.
\]
Then with $x^\ast_i := x^\ast(b_i,\lambda_i,r_i)$, $i=1,2$, we have
\begin{equation}\label{eq:joint_q_perturbation}
\|x^\ast_2-x^\ast_1\|
\!\le\!
\min\!\left\{\!\frac{\sqrt{\bar\tau}}{2\sqrt{\lambda_-}},\frac{\bar\tau\|A\|}{\lambda_-}\!\right\}\!\!
\Bigl(\!\|b_2-b_1\|+\bar y|\lambda_2-\lambda_1|\!\Bigr)  + \bar\tau \|r_2- r_1\|.
\end{equation}
\end{theorem}
\begin{proof}
Since $(b,\lam,r)\mapsto (y^\ast,\tau^\ast)$ is $\cC^\infty$ on $\rr^m\times\rr_{++}\times\rr^n$ by \cref{lem:joint_perturbation_derivatives} and the path $t\mapsto(b(t),\lambda(t),r(t))$ is affine on the compact interval $[0,1]$ with $\lambda(t)\ge\lambda_->0$, the maps $t\mapsto\tau^\ast(t)$ and $t\mapsto\|y^\ast(t)\|$ are continuous on $[0,1]$; hence $\bar\tau$ and $\bar y$ are attained, and finite.

Applying the mean-value inequality along the %
path $t\mapsto \!(b(t),\lambda(t),r(t))$ and setting $x^\ast(t) = x^\ast(b(t),\lambda(t),r(t))$ for brevity, gives
\begin{equation}\label{eq:sup_perturb_bounds}
\begin{aligned}
\|x^\ast_2-x^\ast_1\|
\le&
\sup_{t\in[0,1]}\!\|D_bx^\ast(t)\|\|b_2-b_1\|\! +\!
\sup_{t\in[0,1]}\!\|D_\lambda x^\ast(t)\||\lambda_2-\lambda_1| \\
&+ \sup_{t\in[0,1]}\|D_r x^\ast(t)\|\|r_2-r_1\|.
\end{aligned}
\end{equation}

Next, we bound the partial Jacobians. Define $B:=\sqrt{\tau^\ast}\,A\Diag{p^\ast}^{1/2}$ so that $\tH^*=\lambda I+BB^\top$, and by \cref{lem:joint_perturbation_derivatives}, $D_b x^\ast = \sqrt{\tau^\ast}\,\Diag{p^\ast}^{1/2}B^\top(\lambda I+BB^\top)^{-1}$. Writing the SVD decomposition $B = U \Sigma V\T$, we have
\[
\|B\T(\lambda I+BB\T)^{-1}\| = \|V \Sigma\T(\lambda I+\Sigma \Sigma\T)^{-1} U\T \| \le \max_{t\ge 0} \frac{t}{\lambda+t^2}
= \frac{1}{2\sqrt{\lambda}},
\]
Since $p^\ast\in\Delta_n$, we have $\|\Diag{p^\ast}^{1/2}\|\le 1$, and so $\|D_b x^\ast\| \leq \sqrt{\tau^\ast}/(2 \sqrt{\lambda})$. From \eqref{eq:Dbxstar}, we also have the bound
\[
\|\tau^\ast \Diag{p^\ast}A^\top (\tH^*)^{-1}\| \le \tau^\ast \|\Diag{p^\ast}\| \|A\| \|(\tH^*)^{-1}\| \le \tau^\ast \|A\|/\lambda,
\]
and therefore 
\begin{equation}\label{eq:Dbx-bound}
\|D_b x^\ast\| \le \min\left\{\frac{\sqrt{\tau^\ast}}{2\sqrt{\lambda}},\frac{\tau^\ast\|A\|}{\lambda}\right\}.
\end{equation}
Moreover, immediately from \eqref{eq:Dbxstar}, we have 
\begin{equation}\label{eq:Dlambdax-bound}
    \|D_\lambda x^\ast\|\le \|D_bx^\ast\|\,\|y^\ast\|.
\end{equation}
Finally, notice that $D_r x^\ast = \tau^\ast \Diag{p^\ast}^{1/2}C\Diag{p^\ast}^{1/2}$ where $C := I-B^\top(\lambda I+BB^\top)^{-1}B$ is positive definite with $C \preceq I$. Hence,
\begin{equation}\label{eq:Drx-bound}
\|D_r x^\ast\|\le \tau^\ast \|\Diag{p^\ast}^{1/2}\|^2 \le \tau^\ast.
\end{equation}
Since $\lambda(t)\ge \lambda_-$, combining \eqref{eq:sup_perturb_bounds} with the partial Jacobian bounds \eqref{eq:Dbx-bound}, \eqref{eq:Dlambdax-bound}, and \eqref{eq:Drx-bound} yields the desired bound \eqref{eq:joint_q_perturbation}.
\end{proof}

One may interpret the bound \eqref{eq:joint_q_perturbation} in two regimes: for small~$\lam$, the minimum is attained by $\sqrt{\btau}/(2\sqrt{\lam})$; for large~$\lam$, the bound $\btau\norm{A}/\lam$ is tighter.

\Cref{thm:joint_perturbation} controls $x^\ast(b,\lam,r)$ only for $\lam$ bounded away from zero, where the bound diverges like $\lam^{-1/2}$. At the boundary, the regularization path nevertheless terminates cleanly: with $\tau>0$ and $r\in\rr^n$ held fixed, the iterate converges along $\lam\downarrow 0$ to a unique limit (\cref{thm:lam convergence}).

For this limit analysis, write $P_\lam(x):=f(x)+\lam^{-1}h(x)$, where $f(x):=g_q(x)+\bip{c}{x}$ and $h(x):=\half\norm{Ax-b}^2+\del_{\tau\Deln}(x)$, so that $P_\lam(x)=\phi_p(x,\tau)$. By~\eqref{eq:phi primal}, $P_\lam$ is strictly convex and coercive, so for each $\lam>0$ the problem $\min_x P_\lam(x)$ has a unique minimizer $\hx(\lam)\in\Rn$. Compactness of $\tau\Deln$ implies $S_\tau:=\argmin\rset{\half\norm{Ax-b}^2}{x\in\tau\Deln}\ne\emptyset$, so $\min_{x\in S_\tau}f(x)$ has a unique minimizer $\hx(0)$.

\begin{theorem}[Regularization-path limit]\label{thm:lam convergence}
As $\lam\downarrow 0$,
\begin{enumerate}[(i)]
\item\label{item:f-limit} $f(\hx(\lam))\nearrow f(\hx(0))=\min\rset{f(x)}{x\in S_\tau}$;
\item\label{item:h-limit} $h(\hx(\lam))\searrow \bh:=\min\rset{\half\norm{Ax-b}^2}{x\in\tau\Deln}$;
\item\label{item:x-limit} $\hx(\lam)\rightarrow \hx(0)$.
\end{enumerate}
\end{theorem}

\begin{proof}
Observe that $h(\hx(0))=\bh$, in the notation of \cref{thm:lam convergence}\cref{item:h-limit}.
Let $\lam_1>\lam_2>0$, 
and define $\hh(x):=h(x)-\bh\ge 0$ and $\hP_\lam(x):=f(x)+\lam^{-1}(h(x)-\bh)$
for all $\lam>0$, i.e. $\hP_\lam(x)=P_\lam(x)-\lam^{-1}\bh$.
Then, by the definition of $\hx(\lam)$ for $\lam>0$, we have
\begin{equation}\label{eq:x zero bd}
f(\hx(\lam_1))\le\hP_{\lam_1}(\hx(\lam_1))\le \hP_{\lam_1}(\hx(\lam_2))
\le \hP_{\lam_2}(\hx(\lam_2)) 
\le \hP_{\lam_2}(\hx(0))= f(\hx(0)),
\end{equation}
since $\lam_1^{-1}<\lam_2^{-1}$ and $\hh(x)\ge 0$ for all $x\in\Rn$.
Therefore, $\hP_{\lam}(\hx(\lam))$ is a non-increasing function of $\lam>0$
that is bounded above by $f(\hx(0))$ as $\lam\downarrow 0$. Therefore, there is 
a $\widetilde P\le f(\hx(0))$ such that 
\begin{equation}\label{eq:tP ne}
\hP_{\lam}(\hx(\lam))\nearrow \widetilde P\quad \text{as}\quad 
\lam\downarrow 0.
\end{equation}

Next, add the inequalities
\begin{center}
$
\hP_{\lam_1}(\hx(\lam_1))\le \hP_{\lam_1}(\hx(\lam_2))\quad\text{and}\quad
\hP_{\lam_2}(\hx(\lam_2)) \le \hP_{\lam_2}(\hx(\lam_1))
$
\end{center}
to find that
\begin{center}
$
(\frac{1}{\lam_2}-\frac{1}{\lam_1})\hh(\hx(\lam_2))
=\hP_{\lam_2}(\hx(\lam_2))-\hP_{\lam_1}(\hx(\lam_2))
\le
\hP_{\lam_2}(\hx(\lam_1))-\hP_{\lam_1}(\hx(\lam_1))
=
(\frac{1}{\lam_2}-\frac{1}{\lam_1})\hh(\hx(\lam_1)).
$
\end{center}
Therefore, $\hh(\hx(\lam_2))\le \hh(\hx(\lam_1))$. That is $\hh(\hx(\lam))$
is nondecreasing as a function of $\lam>0$, and it is bounded below by $0$
as $\lam\downarrow 0$.
Consequently, there is a $\tilde h\ge 0$ such that 
\begin{equation}\label{eq:th se}
\hh(\hx(\lam))\searrow \tilde h\ge 0\quad \text{as} \quad \lam\downarrow 0.
\end{equation}

Finally, note that since 
\[\begin{aligned}
f(\hx(\lam_1))+\lam_1^{-1}\hh(\hx(\lam_1))&=
\hP_{\lam_1}(\hx(\lam_1))\le \hP_{\lam_1}(\hx(\lam_2))
\\ &=f(\hx(\lam_2))+\lam_1^{-1}\hh(\hx(\lam_2))
\le 
f(\hx(\lam_2))+\lam_1^{-1}\hh(\hx(\lam_1)),
\end{aligned}\]
we have $f(\hx(\lam_1))\le f(\hx(\lam_2))\le \hP_{\lam_2}(\hx(\lam_2))\le f(\hx(0))$
by \eqref{eq:x zero bd}.
Hence $f(\hx(\lam))$ is a bounded non-increasing function of $\lam>0$.  
Consequently,
there is a $\tilde f\le f(\hx(0))$ such that
\begin{equation}\label{eq:f ne}
f(\hx(\lam))\nearrow \tilde f\le f(\hx(0))\quad \text{as}\quad \lam\downarrow 0.
\end{equation}

By \eqref{eq:tP ne}, \eqref{eq:th se} and \eqref{eq:f ne}, we have
\[
0\le \tilde h=\lim_{\lam\downarrow 0}\hh(\hx(\lam))
=\lim_{\lam\downarrow 0}\lam[\hP_\lam(\hx(\lam))-f(\hx(\lam))]=0.
\]
Hence every cluster point $\tilde x$ of $\{\hx(\lam)\}_{\lam>0}$ is necessarily
an element $S_\tau$ since $0=h(\tilde x)-\bh$. Moreover, 
$f(\tilde x)\le f(\hx(0))=\min\rset{f(x)}{x\in S_\tau}$. The uniqueness of $\hx(0)$
implies that $\tilde x=\hx(0)$. Finally, by \eqref{eq:x zero bd}, observe that 
$\{\hx(\lam)\}_{\lam>0}\subset\lev{f}{f(\hx(0))}$ which is compact. Consequently,
$\hx(\lam)\rightarrow \hx(0)$ since $\hx(0)$ is the unique cluster point.
\end{proof}

\section{Numerical experiments}\label{sec:experiments}

\begin{figure}[t]
  \centering
  \includegraphics[width=\linewidth]{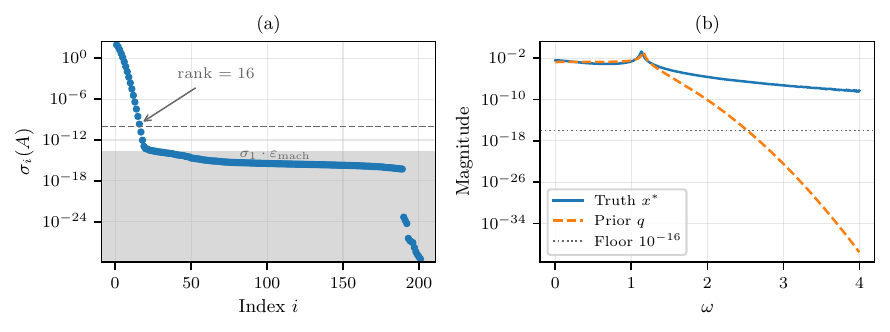}
  \caption{Structure of the UEG test problem. (a)~Singular values of the periodic Laplace kernel $A \in \rr^{201\times 500}$: numerical rank $16$ at tolerance $10^{-10}$; the shaded band marks values below $\sigma_1\,\varepsilon_{\mathrm{mach}}$. (b)~Ground-truth spectral density $x^\ast$ and prior $q$.}
  \label{fig:problem_data}
\end{figure}

Our test problem is a uniform-electron-gas (UEG) analytic-continuation instance from~\citet{chunaDualFormulationMaximum2025}. The task is to recover a nonnegative spectral density $x^\ast$ from noisy imaginary-time observations $b\approx Ax^\ast$ by solving~\eqref{eq:P_problem}. The forward operator is a discretized periodic Laplace transform,
\[
  A_{ij} = e^{-t_i\omega_j} + e^{-(\beta-t_i)\omega_j},
\]
with inverse temperature $\beta\approx 18.68$, $m=201$ imaginary-time samples $t_i\in[0,\beta]$, and $n=500$ frequency nodes $\omega_j\in[4\times 10^{-3},\,4]$. Exponential decay in the kernel entries makes the map $A$ severely ill-conditioned;
see \cref{fig:problem_data}(a).

The ground truth $x^\ast$ is the Mermin function: its shape $p^\ast:=x^\ast/\tau^\ast\in\Delta_n$ carries physical meaning, whereas its normalization $\tau^\ast=\ip{\ones,x^\ast}$ is unknown \emph{a priori}. The prior $q\in\Delta_n$ is the random-phase approximation (RPA), a standard model for the electron-liquid response that specifies spectral shape but not scale; the peak near $\omega\approx 1$ in \cref{fig:problem_data}(b) is the plasmon resonance~\cite[\S 5.3]{giuliani2008quantum}. This is precisely the setting that motivates the scale-shape decomposition of \cref{sec:new_approach}: the physics supplies a shape prior on $\Delta_n$, and the scale $\tau^\ast$ must be determined by the data.

The observations are corrupted by multiplicative noise at relative level $10^{-4}$, and we set $c=0$ and $\lambda=10^{-5}$ throughout. Because \cref{alg:inexact Newton} requires $q\in\rr^n_{++}$ and the prior has an exponentially small tail (half of entries fall below $10^{-10}$, with a minimum of $2.8\times 10^{-40}$), we clip $q\leftarrow\max(q,\,10^{-16})$ and renormalize before the first iteration. Three experiments follow. \Cref{subsec:overflow} sweeps $Z:=\tau^\ast\in\set{2^4,2^6,2^8,2^{10}}$ to probe overflow resilience, \cref{subsec:analytic_continuation} varies $Z\in\set{10^0,10^1,10^2}$ to probe scale recovery from a fixed initialization $\tau_0=1$, and \cref{subsec:regularization_path} sweeps $\lambda$ to trace the regularization path predicted by \cref{thm:joint_perturbation,thm:lam convergence}.

\paragraph{Algorithm settings} Across all experiments we run \cref{alg:inexact Newton} from $y^0=0$ with exact Newton steps ($\eta_k\equiv 0$, solved by a dense symmetric-indefinite factorization), Armijo constant $\mu=0.49$ (near the upper bound $1/2$ of \cref{rem:merit fnc}), and backtracking factor $\gam=\half$. The level-set parameter is $\beta = 1.5\,\rho(z^0)$, and the safeguard $\htaubetamin$ uses the closed-form bound $\lambda W(\zeta)/A_{\max}^2$ from \cref{lem:level cpt}\ref{item:taubetamin}, floored at $10^{-16}$. The merit tolerance and iteration cap are $(\eps,K_{\max})=(10^{-8},300)$ throughout.

\begin{figure}[t]
  \centering
  \includegraphics[width=0.85\linewidth]{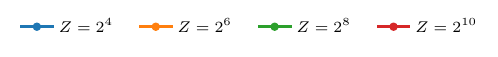}\\[-0.6em]
  \begin{subfigure}[t]{0.325\linewidth}
    \centering
    \includegraphics[width=\linewidth]{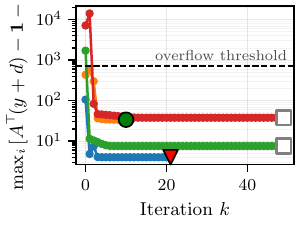}
    \caption{Standard dual Newton}\label{fig:traj:a}
  \end{subfigure}\hfill
  \begin{subfigure}[t]{0.325\linewidth}
    \centering
    \includegraphics[width=\linewidth]{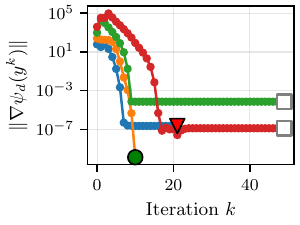}
    \caption{Standard\@ dual Newton}\label{fig:traj:b}
  \end{subfigure}\hfill
  \begin{subfigure}[t]{0.325\linewidth}
    \centering
    \includegraphics[width=\linewidth]{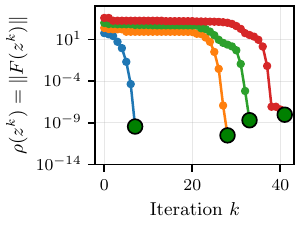}
    \caption{Scale-shape dual Newton}\label{fig:traj:c}
  \end{subfigure}
  \\[-1.0em]
  \includegraphics[width=0.55\linewidth]{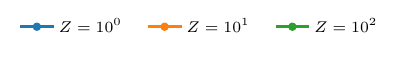}\\[-0.8em]
  \begin{subfigure}[t]{0.325\linewidth}
    \centering
    \includegraphics[width=\linewidth]{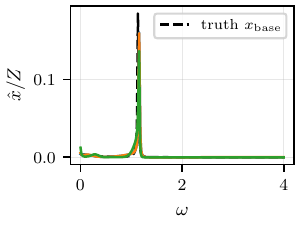}
    \caption{Computed solutions}
  \end{subfigure}\hfill
  \begin{subfigure}[t]{0.325\linewidth}
    \centering
    \includegraphics[width=\linewidth]{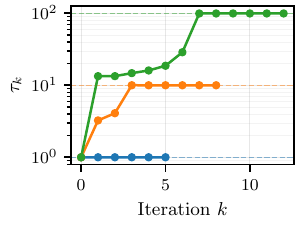}
    \caption{Scale trajectory}\label{fig:traj:e}
  \end{subfigure}\hfill
  \begin{subfigure}[t]{0.325\linewidth}
    \centering
    \includegraphics[width=\linewidth]{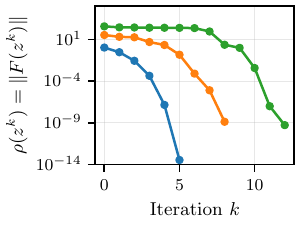}
    \caption{Merit convergence}\label{fig:traj:f}
  \end{subfigure}
  \caption{Iteration trajectories at $\lambda = 10^{-5}$, $\tau_0 = 1$. \emph{Top row, overflow resilience} (\cref{subsec:overflow}), $Z\in\set{2^4,2^6,2^8,2^{10}}$: (a) maximum componentwise argument $\max_j(A\T y^k - \ones - c)_j$ of the exponential in $\nabla\psi_d$ under full (undamped) Newton steps. (b)~classical dual Newton gradient $\norm{\nabla\psi_d(y^k)}$ with Armijo backtracking plus overflow safeguard. (c)~\cref{alg:inexact Newton} merit $\rho(z^k)=\norm{F(z^k)}$. \emph{Bottom row, scale recovery} (\cref{subsec:analytic_continuation}), $Z\in\set{10^0,10^1,10^2}$: (d) converged recoveries $\hat{x}/Z$ overlaid on the truth shape; (e) scale trajectories $\tau_k$ with targets $\tau^\ast = Z$ (dashed); (f) merit $\rho(z^k)$. Endpoint markers: {\large$\bullet$}~converged, {\large$\blacktriangledown$}~Armijo failure, {\large$\square$}~iteration budget exceeded.}
  \label{fig:trajectories}
\end{figure}

\subsection{Overflow resilience}\label{subsec:overflow}

Any first- or second-order method applied to the unnormalized dual~\eqref{eq:original_dual_obj} must evaluate $\vexp(A\T y - c)$, which can overflow when entries of $A\T y - c$ are large. The scale-shape reformulation~\eqref{eq:dual_func} replaces these bare exponentials with $\logexp_q$ and $\dlogexp$, which admit the standard max-shift evaluation~\citep{blanchardAccuratelyComputingLogsumexp2021} applied to the $q$-weighted variants; every exponent in the evaluation of $F$ and $DF$ in \cref{alg:inexact Newton} is then nonpositive.

The benchmark comparator is Newton's method applied directly to the standard dual~\eqref{eq:original_dual_obj} with Armijo backtracking, augmented by an overflow safeguard that halves the step whenever the trial exponent would exceed the \texttt{float64} range. Without this safeguard, a textbook implementation fails outright once the full Newton step crosses the overflow threshold. This safeguarded Newton method therefore measures the performance penalty of avoiding overflow, rather than simply crashing. The same pathology affects any method that requires unnormalized exponentials, such as \citet{decarreau_DualMethodsEntropy_1992}.

Set $x^\ast \!=\! Zp^\ast$ with $p^\ast \in \Delta_n$ and $\tau_0 = 1$. Sweep the target scale $Z \in \set{2^4, 2^6, 2^8, 2^{10}}$. \Cref{fig:trajectories}(a) plots the unsafeguarded full-Newton-step exponent $\max_i[A\T(y^k + \Delta y^k) - \ones - c]_i$ for the comparator on a log scale. The curves separate by $Z$: the exponent stays below $600$ for $Z \in \set{2^4, 2^6}$, below the float64 threshold $\log(\texttt{float64\_max}) \approx 709$ (dashed). However, for $Z \ge 2^8$ the exponent exceeds the overflow threshold by more than an order of magnitude. The safeguarded backtracking halves the step several times at iteration $k = 0$ (where $y_k=0$) before accepting a valid trial point.

Panels~(b)~and~(c) of \cref{fig:trajectories} compare the resulting convergence. Standard dual Newton reaches $\norm{\nabla\psi_d(y^k)} \le 10^{-8}$ only for $Z = 2^6$ ($11$ iterations). At $Z = 2^4$ the Armijo backtracking stalls at $\norm{\nabla\psi_d(y^k)} \approx 10^{-7}$ with no overflow triggered. This failure reflects the precision floor imposed by the ill-conditioning of the dual Hessian $\nabla^2\psi_d$. At $Z \in \set{2^8, 2^{10}}$ the overflow-aware backtracking allows for the first Newton steps, but the gradient norm does not reach the required tolerance of $10^{-8}$ within the iteration limit.

By contrast, \cref{alg:inexact Newton} reaches the merit threshold tolerance $\rho \le 10^{-8}$ well before the iteration cap, across the full $Z$ range. The $Z = 2^{6}$ and larger trajectories exhibit the expected rate transition from linear to quadratic convergence (\cref{thm:convergence of inexact Newton 1}).

\subsection{Scale recovery}\label{subsec:analytic_continuation}

Set $x^\ast = Z p^\ast$ with $p^\ast \in \Delta_n$ and sweep $Z \in \set{10^0,\,10^1,\,10^2}$. This experiment tests scale recovery: \cref{alg:inexact Newton} must find $\tau^\ast = Z$ from the initialization $\tau_0 = 1$.

\Cref{alg:inexact Newton} converges in every case. For $Z = 1$, the terminal scale is $\tau_K = 1.0002$ after $6$ iterations; for $Z = 10^1$, $\tau_K = 10.003$ after $9$ iterations; for $Z = 10^2$, $\tau_K = 99.46$ after $13$ iterations. The relative errors $|\tau_K - Z|/Z$ are below $6\times 10^{-3}$ in all cases, and the final quadratic step drives the merit function $\rho$ well below the tolerance $\eps=10^{-8}$.

\Cref{fig:trajectories}(d) shows that the normalized recoveries $x^\ast/Z$ coincide across~$Z$ and illustrates the scale-shape decomposition. Panel~(e) traces $\tau_k$ as it increases from $\tau_0 = 1$ towards its target. Panel~(f) shows the linear-to-quadratic convergence transition of \cref{thm:convergence of inexact Newton 1}.

\subsection{Regularization path}\label{subsec:regularization_path}

\begin{figure}[t]
  \centering
  \includegraphics[width=0.8\linewidth]{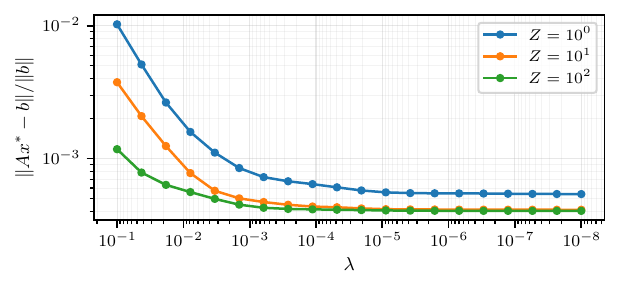}
  \caption{Relative data-fit residual $\norm{Ax^\ast - b}/\norm{b}$ versus $\lambda$ (log--log) for $Z \in \set{10^0,\,10^1,\,10^2}$. Each point is an independent solve from $\tau_0 = 1$.}
  \label{fig:regularization_path}
\end{figure}

\Cref{fig:regularization_path} traces the relative data-fit residual $\norm{Ax^\ast - b}/\norm{b}$ as $\lambda$ sweeps from $10^{-1}$ down to $10^{-8}$ for each scale $Z \in \set{10^0,\,10^1,\,10^2}$. The residual is large for $\lambda\gg 1$, where the entropy term dominates, and decays as the balance shifts toward data fidelity, plateauing at a noise-limited floor near $\lambda\approx 10^{-5}$.

By \cref{lem:joint_perturbation_derivatives}, the solution map $(b,\lam,r)\mapsto x^\ast$ is $\mathcal{C}^\infty$ on $\rr^m\times\rr_{++}\times\rr^n$, so $\lam\mapsto x^\ast(\lam)$ is smooth for $\lam>0$; \cref{thm:lam convergence} governs the $\lam\downarrow 0$ limit, in which $x^\ast(\lam)$ approaches the constrained least-squares solution on $\tau\Deln$ and $\norm{Ax^\ast(\lam)-b}$ decreases monotonically to the constrained-LS residual, given here by the noise floor. The steep initial drop is quantified by \cref{thm:joint_perturbation}: fixing $b$ and $q$,
\[
  \norm{x^\ast(\lam_2)-x^\ast(\lam_1)} \;\le\; \min\Set{\frac{\sqrt{\btau}}{2\sqrt{\lam_-}},\;\frac{\btau\norm{A}}{\lam_-}}\,\bar y\,|\lam_2-\lam_1|,
\]
where $\lam_-:=\min\set{\lam_1,\lam_2}$ and $\btau$, $\bar y$ are the path maxima of \cref{thm:joint_perturbation}. In the experimental window $\btau$ stays bounded as $\lam_-\downarrow 0$ (the iterates satisfy $\tau_K\approx Z$, see \cref{subsec:analytic_continuation}), while $F_y(y^\ast,\tau^\ast)=0$ gives the dual bound $\bar y\le(\norm{b}+A_{\max}\btau)/\lam_-$. Combined with the $\sqrt{\btau}/(2\sqrt{\lam_-})$ prefactor of~\eqref{eq:joint_q_perturbation}, the Lipschitz modulus grows as $O(\lam_-^{-3/2})$ as $\lam_-\downarrow 0$, matching the steep decay in \cref{fig:regularization_path}.

The three residual curves do not collapse under the $\norm{b}$-normalization. The $Z=1$ curve saturates near $5\times 10^{-4}$, while $Z=10$ and $Z=100$ both saturate near $2\times 10^{-4}$. This reflects a scale-dependent balance in the primal objective~\eqref{eq:P_problem}: the data-fit term $(2\lam)^{-1}\norm{Ax-b}^2$ scales as $\Theta(Z^2)$ and the entropy $g_q(x)$ as $\Theta(Z\log Z)$. Larger $Z$ shifts the balance toward data fidelity, as observed.

\Cref{alg:inexact Newton} reaches $\rho\le 10^{-10}$ in $6$--$30$ iterations for $\lam\ge 10^{-6}$ across the $(\lam,Z)$ grid. \Cref{lem:DF inv bd}\cref{item:singular-bounds} predicts this slowdown: $\sig_{\min}(DF(z))\ge \min\set{\lam,\,1/\tau}$ gives $\norm{DF(z)^{-1}}=O(\lam^{-1})$ for bounded $\tau$, so the linear contraction rate of \cref{thm:convergence of inexact Newton 1} deteriorates as $\lam\downarrow 0$.

\section*{Acknowledgments}
The authors thank Thomas Chuna for supplying the UEG data for the experiments of \cref{sec:experiments}. JB and MPF thank the Centre de recherches math\'ematiques and the Simons Foundation for research support during Summer 2025.
\clearpage
\bibliographystyle{siamplain}
\bibliography{shorttitles,refs}

\end{document}